\documentclass[10pt]{amsart}

\usepackage[left=2.8cm, right=2.8cm, top=2.2cm, bottom=2.2cm]{geometry}

\usepackage{amsmath,amssymb,amsthm,mathrsfs,stmaryrd,latexsym}
\usepackage{amsfonts}
\usepackage{accents}

\usepackage{tcolorbox}
\tcbuselibrary{skins, breakable, theorems}

\usepackage{enumerate} 
\usepackage{tikz}
\usepackage{etoolbox}

\def\Luoma#1{\uppercase\expandafter{\romannumeral#1}}
\def\luoma#1{\romannumeral#1}

\usepackage{url}

\newtheorem{mythm}{Theorem}[section]
\newtheorem{mylem}[mythm]{Lemma}
\newtheorem{myprop}[mythm]{Proposition}
\newtheorem{mycor}[mythm]{Corollary}

\theoremstyle{definition}
\newtheorem{mydefn}[mythm]{Definition}
\newtheorem{myexample}[mythm]{Example}

\theoremstyle{remark}
\newtheorem{myrem}[mythm]{Remark}
\newtheorem{mypara}[mythm]{}

\usepackage[all]{xy}
\usepackage{graphicx}
\usepackage{subfigure}

\newcommand{\bb}{\mathbb}
\newcommand{\ca}{\mathcal}
\newcommand{\ak}{\mathfrak}
\newcommand{\scr}{\mathscr}

\newcommand{\mbf}{\mathbf}
\newcommand{\mrm}{\mathrm}

\def\op#1{\mathop{\mathrm{#1}}}

\newcommand{\ho}{\mrm{Hom}}

\newcommand{\ke}{\mrm{Ker}}
\newcommand{\cok}{\mrm{Coker}}
\newcommand{\im}{\mrm{Im}}
\newcommand{\df}{\mrm{d}}
\newcommand{\id}{\mrm{id}}
\newcommand{\tot}{\mrm{Tot}}
\newcommand{\spec}{\op{Spec}}
\newcommand{\colim}{\op{colim}}

\newcommand{\ob}{\mrm{Ob}}
\newcommand{\rr}{\mrm{R}}
\newcommand{\dl}{\mrm{L}}

\newcommand{\iso}{\stackrel{\sim}{\longrightarrow}}

\newcommand{\dd}{\mbf{D}}

\title{Almost Coherence of Higher Direct Images}
\author{Tongmu He}
\date{\today}
\address{Tongmu He, Institut des Hautes \'Etudes Scientifiques, 35 route de Chartres, 91440 Bures-sur-Yvette, France}
\email{hetm15@ihes.fr}

\usepackage{hyperref}
\hypersetup{colorlinks,
	citecolor=black,
	filecolor=black,
	linkcolor=black,
	urlcolor=black,
	pdftitle={Almost Coherence of Higher Direct Images},
	pdfauthor={Tongmu He},
	pdftex}

\numberwithin{equation}{mythm}

\begin{document}
	\maketitle
	
\begin{abstract}
	For a flat proper morphism of finite presentation between schemes with almost coherent structural sheaves (in the sense of Faltings), we prove that the higher direct images of quasi-coherent and almost coherent modules are quasi-coherent and almost coherent. Our proof uses Noetherian approximation, inspired by Kiehl's proof of the pseudo-coherence of higher direct images. Our result allows us to extend Abbes-Gros' proof of Faltings' main $p$-adic comparison theorem in the relative case for projective log-smooth morphisms of schemes to proper ones, and thus also their construction of the relative Hodge-Tate spectral sequence.
\end{abstract}

	
	\tableofcontents

\section{Introduction}

\begin{mypara}
One of the first important results in algebraic geometry is the fact that the coherence for modules is preserved by higher direct images by a proper morphism. The Noetherian case is due to Grothendieck \cite[3.2.1]{ega3-1}, and the general case is due to Kiehl \cite[2.9']{kiehl1972finite}. The goal of this article is to extend the following corollary to almost algebra, motivated by applications in $p$-adic Hodge theory.
\end{mypara}

\begin{mythm}[{Kiehl \cite[2.9']{kiehl1972finite}, see \cite[1.4.8]{abbes2010rigid}}]\label{thm:intro-rigid}
	Let $f:X\to S$ be a morphism of schemes satisfying the following conditions:
	\begin{enumerate}
		\renewcommand{\labelenumi}{{\rm(\theenumi)}}
		\item $f$ is proper and of finite presentation, and
		\item $\ca{O}_S$ is universally coherent.
	\end{enumerate}
	Then, for any coherent $\ca{O}_X$-module $\ca{M}$ and any $q\in\bb{N}$, $\rr^q f_*\ca{M}$ is a coherent $\ca{O}_S$-module.
\end{mythm}
We say that $\ca{O}_S$ is \emph{universally coherent} if there is a covering $\{S_i=\spec(A_i)\}_{i\in I}$ of $S$ by affine open subschemes such that the polynomial algebra $A_i[T_1,\dots,T_n]$ is a coherent ring for any $i\in I$ and $n\in\bb{N}$. Indeed, such a condition on $\ca{O}_S$ implies that the coherent $\ca{O}_X$-module $\ca{M}$ is actually \emph{pseudo-coherent} relative to $S$, which roughly means that if we embed $X$ locally as a closed subscheme of $\bb{A}^n_{S_i}$, then $\ca{M}$ admits a resolution by finite free modules over $\bb{A}^n_{S_i}$. Theorem \ref{thm:intro-rigid} is a direct corollary of Kiehl's result \cite[2.9']{kiehl1972finite}, saying that the derived pushforward $\rr f_*$ sends a relative pseudo-coherent complex to a pseudo-coherent complex.

\begin{mypara}
	Almost algebra was introduced by Faltings \cite{faltings1988p, faltings2002almost} for the purpose of developing $p$-adic Hodge theory. The setting is a pair $(R,\ak{m})$ consisting of a ring $R$ with an ideal $\ak{m}$ such that $\ak{m}=\ak{m}^2$, and the rough idea is to replace the category of $R$-modules by its quotient by $\ak{m}$-torsion modules. An ``almost'' analogue of Theorem \ref{thm:intro-rigid} is necessary for Faltings' approach to $p$-adic Hodge theory. Indeed, under the same assumptions of \ref{thm:intro-rigid}, Abbes-Gros \cite[2.8.14]{abbes2020suite} proved that $\rr^qf_*$ sends a quasi-coherent and almost coherent $\ca{O}_X$-module to a quasi-coherent and almost coherent $\ca{O}_S$-module, by reducing directly to \ref{thm:intro-rigid}. This result plays a crucial role in the proof of Faltings' main $p$-adic comparison theorem in the absolute case (see \cite[4.8.13]{abbes2020suite}), and thus of the Hodge-Tate decomposition (see \cite[6.4.14]{abbes2020suite}). Later, Zavyalov \cite[5.1.6]{zavyalov2021almost} extended the same almost coherence result to formal schemes.
	
	However, the almost coherence result \cite[2.8.14]{abbes2020suite} is not enough for Faltings' main $p$-adic comparison theorem in the relative case (thus neither for the relative Hodge-Tate spectral sequence), since we inevitably encounter the situation where $\ca{O}_S$ is \emph{universally almost coherent} but not universally coherent. Thus, under the assumptions that
	\begin{enumerate}
		\renewcommand{\labelenumi}{{\rm(\theenumi)}}
		\item $f$ is projective, flat and of finite presentation, and that
		\item $\ca{O}_S$ is universally almost coherent,
	\end{enumerate}
	Abbes-Gros proved an almost coherence result \cite[2.8.18]{abbes2020suite} by adapting the arguments of \cite[\Luoma{3}.2.2]{sga6}, where the projectivity condition on $f$ plays a crucial role. This is the reason why Faltings' main $p$-adic comparison theorem in the relative case (and thus the relative Hodge-Tate spectral sequence) was only proved for projective log-smooth morphisms in \cite[5.7.4 (and 6.7.5)]{abbes2020suite}.
\end{mypara}

\begin{mypara}
	In this article, we generalize the almost coherence result \cite[2.8.18]{abbes2020suite} to proper morphisms, which allows us to extend Abbes-Gros' proof of Faltings' main $p$-adic comparison theorem in the relative case to proper log-smooth morphisms, and thus also their construction of the relative Hodge-Tate spectral sequence (see Section \ref{sec:remark}).
	
	Let $R$ be a ring with an ideal $\ak{m}$ such that for any integer $l\geq 1$, the $l$-th powers of elements of $\ak{m}$ generate $\ak{m}$. The pair $(R,\ak{m})$ will be our basic setup for almost algebra (see Section \ref{sec:coh}). The main theorem of this article is the following
\end{mypara}

\begin{mythm}[see \ref{thm:main}]\label{thm:intro-main}
	Let $f:X\to S$ be a morphism of $R$-schemes satisfying the following conditions:
	\begin{enumerate}
		\renewcommand{\labelenumi}{{\rm(\theenumi)}}
		\item $f$ is proper, flat and of finite presentation, and
		\item $\ca{O}_X$ and $\ca{O}_S$ are almost coherent.
	\end{enumerate}  
	Then, for any quasi-coherent and almost coherent $\ca{O}_X$-module $\ca{M}$ and any $q\in\bb{N}$, $\rr^q f_*\ca{M}$ is a quasi-coherent and almost coherent $\ca{O}_S$-module.
\end{mythm}

Our proof is close to Kiehl's proof of \cite[2.9]{kiehl1972finite}, see Section \ref{sec:proof}. We roughly explain our ideas in the following:
\begin{enumerate}
	\renewcommand{\labelenumi}{{\rm(\theenumi)}}
	\item We may assume without loss of generality that $S$ is affine. As $\ca{M}$ is not of finite presentation over $X$ in general, we couldn't descend it by Noetherian approximation. But $\ca{M}$ is almost coherent, for any $\pi\in\ak{m}$ we can ``$\pi$-resolve'' $\ca{M}$ over a truncated \v Cech hypercovering $X_\bullet=(X_n)_{[n]\in\Delta_{\leq k}}$ (by affine open subschemes) of $X$ by finite free modules $\ca{F}_\bullet^\bullet$ as in \cite[2.2]{kiehl1972finite}, where each $\ca{F}_n^\bullet$ is a ``resolution'' of $\ca{M}|_{X_n}$ modulo $\pi$-torsion, see Section \ref{sec:coh}. By Noetherian approximation, we obtain a proper flat morphism $f_\lambda:X_\lambda\to S_\lambda$ of Noetherian schemes together with a complex of finite free modules $\ca{F}_{\lambda,\bullet}^\bullet$ over a truncated \v Cech hypercovering of $X_\lambda$ descending $f$ and $\ca{F}_\bullet^\bullet$.
	\item As in \cite[1.4]{kiehl1972finite}, the descent data of $\ca{M}$ over $X_\bullet$ are encoded as null homotopies of the multiplication by a certain power of $\pi$ on the cone of $\alpha^*\ca{F}_{m}^\bullet\to \ca{F}_n^\bullet$ (where $\alpha:[m]\to [n]$ is a morphism in the truncated simplicial category $\Delta_{\leq k}$), see Section \ref{sec:pi-isom}. We can descend the latter by Noetherian approximation, from which we produce some coherent modules over $X_\lambda$, see Section \ref{sec:glue}.
	\item Applying the classical coherence result for $f_\lambda:X_\lambda\to S_\lambda$, we see that the \v Cech complex of $\ca{F}_{\lambda,\bullet}^\bullet$ is ``pseudo-coherent'' modulo certain power of $\pi$, see Section \ref{sec:ps-coh}. The same thing holds for the \v Cech complex of $\ca{F}_{\bullet}^\bullet$ by base change (due to the flatness of $f_\lambda$). Since this \v Cech complex computes $\rr\Gamma(X,\ca{M})$ up to certain degree and modulo certain power of $\pi$, the conclusion follows by varying $\pi$ in $\ak{m}$.
\end{enumerate}

\subsection*{Acknowledgements}
This work was completed while I was a doctoral student at Universit\'e Paris-Saclay and Institut des Hautes \'Etudes Scientifiques, under the supervision of Ahmed~ Abbes. I am deeply grateful to him for posing the question, conducting a meticulous review of this work, and offering numerous helpful suggestions. Additionally, I would like to extend my appreciation to the anonymous referee for their incredibly thorough examination and detailed feedback, which have further enhanced the clarity and depth of this article.

\section{Notation and Conventions}
\begin{mypara}
	All rings considered in this article are unitary and commutative.
\end{mypara}

\begin{mypara}\label{para:notation-complex}
	Let $\scr{A}$ be an abelian category. When we consider ``a complex in $\scr{A}$'', we always refer to a cochain complex in $\scr{A}$, and we denote it by $M^\bullet$ with differential maps $\df^n:M^n\to M^{n+1}$ ($n\in\bb{Z}$). For any $a\in \bb{Z}$, we denote by $\tau^{\geq a}M^\bullet$ (resp. $\sigma^{\geq a}M^\bullet$) the canonical (resp. stupid) truncation of $M^\bullet$, see \cite[\href{https://stacks.math.columbia.edu/tag/0118}{0118}]{stacks-project}.
\end{mypara}

\begin{mypara}
	Let $\Delta$ be the category formed by finite ordered sets $[n]=\{0,1,\dots,n\}$ ($n\in\bb{N}$) with non-decreasing maps (\cite[\href{https://stacks.math.columbia.edu/tag/0164}{0164}]{stacks-project}). For $k\in \bb{N}\cup\{\infty\}$, we denote by $\Delta_{\leq k}$ the full subcategory of $\Delta$ formed by objects $[0],[1],\dots,[k]$. For a category $C$, a contravariant functor from $\Delta_{\leq k}$ to $C$ sending $[n]$ to $X_n$ is called a \emph{$k$-truncated simplicial object} of $C$, denoted by $X_\bullet$. Let $\ca{P}$ be a property for objects of $C$. We say that $X_\bullet$ has property $\ca{P}$ if each $X_n$ has property $\ca{P}$.
\end{mypara}

\section{Isomorphisms up to Bounded Torsion}\label{sec:pi-isom}
In this section, we fix a ring $R$ and an element $\pi$ of $R$. Consider an abelian category $\scr{A}$ with a ring homomorphism $R\to \mrm{End}(\id_{\scr{A}})$, where $\mrm{End}(\id_{\scr{A}})$ is the ring of endomorphisms of the identity functor. Thus, $\pi$ defines a functorial endomorphism on any object $M$ of $\scr{A}$. We denote by $\mbf{K}(\scr{A})$ the homotopy category of complexes in $\scr{A}$.
\begin{mydefn}\label{defn:pi-iso}
	\begin{enumerate}
		\renewcommand{\labelenumi}{{\rm(\theenumi)}}
		\item We say that an object $M$ in $\scr{A}$ is \emph{$\pi$-null} if it is annihilated by $\pi$. We say that a morphism $f:M\to N$ in $\scr{A}$ is a \emph{$\pi$-isomorphism} if its kernel and cokernel are $\pi$-null.
		\item We say that a complex $M^\bullet$ in $\scr{A}$ is \emph{$\pi$-exact} if the cohomology group $H^n(M^\bullet)$ is $\pi$-null for any $n\in\bb{Z}$. We say that a morphism of complexes $f:M^\bullet\to N^\bullet$ in $\scr{A}$ is a \emph{$\pi$-quasi-isomorphism} if it induces a $\pi$-isomorphism on the cohomology groups $H^n(f):H^n(M^\bullet)\to H^n(N^\bullet)$ for any $n\in\bb{Z}$.
	\end{enumerate} 
\end{mydefn}

\begin{mylem}[{\cite[2.6.3]{abbes2020suite}}]\label{lem:pi-iso-retract}
	Let $f:M\to N$ be a morphism in $\scr{A}$.
	\begin{enumerate}
		\renewcommand{\labelenumi}{{\rm(\theenumi)}}
		\item If there exists a morphism $g:N\to M$ in $\scr{A}$ such that $g\circ f=\pi\cdot\id_M$ and $f\circ g=\pi\cdot\id_N$, then $f$ is a $\pi$-isomorphism.\label{item:lem:pi-iso-retract-2}
		\item If $f$ is a $\pi$-isomorphism, then $\pi\cdot\id_N$ and $\pi\cdot\id_M$ uniquely factor through $N\to \im(f)$ and $\im(f)\to M$ respectively, whose composition $g:N\to M$ satisfies that $g\circ f=\pi^2\cdot\id_M$ and $f\circ g=\pi^2\cdot\id_N$. In particular, the morphism $g$ is functorial in $f$.\label{item:lem:pi-iso-retract-1}
	\end{enumerate}
\end{mylem}

\begin{mylem}\label{lem:pi-exact-homotopy}
	Let $f:M^\bullet\to N^\bullet$ be a morphism of complexes in $\scr{A}$. Assume the following conditions:
	\begin{enumerate}
		\renewcommand{\labelenumi}{{\rm(\theenumi)}}
		\item for any $i> 0$, $M^i=N^i=0$;
		\item there is $n\in \bb{N}$ such that for any $-n\leq i\leq 0$, $M^i$ is projective and $\pi\cdot H^i(N^\bullet)=0$.
	\end{enumerate}
	Then, there exists a morphism $s^i:M^i\to N^{i-1}$ for any $-n\leq i\leq 0$, such that 
	\begin{align}\label{eq:lem:pi-exact-homotopy}
		\pi^{1-i}\cdot f^i=\pi\cdot s^{i+1}\circ \df^i+\df^{i-1}\circ s^i
	\end{align}
	as morphisms from $M^i$ to $N^i$, where we put $s^1=0$. In particular, the morphism of canonically truncated complexes
	\begin{align}
		\pi^{n+1}\cdot f:\tau^{\geq -n}M^\bullet \longrightarrow \tau^{\geq -n}N^\bullet
	\end{align}
	is homotopic to zero.
\end{mylem}
\begin{proof}
	We construct $s^i$ by induction. Setting $0=s^1=s^2=\cdots$, we may assume that we have already constructed the homomorphisms for any degree strictly bigger than $i$ with identities \eqref{eq:lem:pi-exact-homotopy}. As $\pi^{-i}\cdot f^{i+1}=\pi\cdot s^{i+2}\circ \df^{i+1}+\df^{i}\circ s^{i+1}$, we see that
	\begin{align}
		\df^{i}\circ(\pi^{-i}\cdot f^i-s^{i+1}\circ\df^{i})=\pi^{-i}\cdot f^{i+1}\circ\df^{i}-(\pi^{-i}\cdot f^{i+1}-\pi\cdot s^{i+2}\circ \df^{i+1})\circ \df^{i}=0.
	\end{align}
	Thus, the map $\pi^{-i}\cdot f^i-s^{i+1}\circ\df^{i}:M^i\to N^i$ factors through $\ke(\df^i:N^i\to N^{i+1})$. The assumption $\pi\cdot H^i(N^\bullet)=0$ implies that the map $\pi^{1-i}\cdot f^i-\pi\cdot s^{i+1}\circ\df^{i}:M^i\to N^i$ factors through $\im(\df^{i-1}:N^{i-1}\to N^{i})$. As $M^i$ is projective, there exists a morphism $s^i:M^i\to N^{i-1}$ such that $\pi^{1-i}\cdot f^i-\pi\cdot s^{i+1}\circ\df^{i}=\df^{i-1}\circ s^i$, which completes the induction.
	In particular, for any $i\geq -n$, we have
	\begin{align}\label{eq:3.3.4}
		\pi^{n+1}\cdot f^i=(\pi^{n+i+1}\cdot s^{i+1})\circ \df^i+\df^{i-1}\circ (\pi^{n+i}\cdot s^i).
	\end{align}
	Recall that $\tau^{\geq -n}M^\bullet=(0\to M^{-n}/\im(\df^{-n-1})\to M^{1-n}\to\cdots \to M^0\to 0)$. Thus, we see that $\pi^{n+1}\cdot f:\tau^{\geq -n}M^\bullet \to \tau^{\geq -n}N^\bullet$ is homotopic to zero.
\end{proof}

\begin{myprop}\label{prop:pi-exact-homotopy}
	Let $P^\bullet$ be a complex of projective objects in $\scr{A}$, and let $M^\bullet$ be a $\pi$-exact complex in $\scr{A}$. Assume that there are integers $a\leq b$ such that $P^i$ and $M^i$ vanish for any $i\notin[a,b]$. Then, the $R$-module $\ho_{\mbf{K}(\scr{A})}(P^\bullet,M^\bullet)$ is $\pi^{b-a+1}$-null.
\end{myprop}
\begin{proof}
	It follows directly from \ref{lem:pi-exact-homotopy}.
\end{proof}

\begin{mycor}\label{cor:pi-exact-homotopy}
	Let $P^\bullet$ be a complex of projective objects in $\scr{A}$, and let $f:M^\bullet\to N^\bullet$ be a $\pi$-quasi-isomorphism of complexes in $\scr{A}$. Assume that there are integers $a\leq b$ such that $P^i$, $M^i$ and $N^i$ vanish for any $i\notin[a,b]$. Then, the map $\ho_{\mbf{K}(\scr{A})}(P^\bullet,M^\bullet)\to \ho_{\mbf{K}(\scr{A})}(P^\bullet,N^\bullet)$ induced by $f$ is a $\pi^{2(b-a+3)}$-isomorphism of $R$-modules.
\end{mycor}
\begin{proof}
	There is an exact sequence of $R$-modules
	\begin{align}\label{eq:cor:pi-exact-homotopy}
		\ho_{\mbf{K}(\scr{A})}(P^\bullet,C^\bullet[-1])\to\ho_{\mbf{K}(\scr{A})}(P^\bullet,M^\bullet)\to\ho_{\mbf{K}(\scr{A})}(P^\bullet,N^\bullet)\to\ho_{\mbf{K}(\scr{A})}(P^\bullet,C^\bullet)
	\end{align}
	where $C^\bullet$ is the cone of $f$ (\cite[\href{https://stacks.math.columbia.edu/tag/0149}{0149}]{stacks-project}). As $C^\bullet$ and $C^\bullet[-1]$ are $\pi^2$-exact and vanish outside $[a-1,b+1]$, the outer two $R$-modules are $\pi^{2(b-a+3)}$-null by \ref{prop:pi-exact-homotopy}, whence we draw the conclusion.
\end{proof}

\begin{mylem}\label{lem:pi-exact-homotopy-lift}
	Let $g:P^\bullet\to N^\bullet$ and $f:M^\bullet\to N^\bullet$ be morphisms of complexes in $\scr{A}$. Assume that there are integers $a\leq b$ such that
	\begin{enumerate}
		\renewcommand{\labelenumi}{{\rm(\theenumi)}}
		\item $M^i=N^i=0$ for any $i>b$, and that\label{item:lem:pi-exact-homotopy-lift-1}
		\item $P^i$ is projective for any $i\in [a,b]$ and zero for any $i\notin [a,b]$, and that\label{item:lem:pi-exact-homotopy-lift-2}
		\item the map $H^i(f):H^i(M^\bullet)\to H^i(N^\bullet)$ is a $\pi$-isomorphism for $i>a$ and $\pi$-surjective for $i=a$.\label{item:lem:pi-exact-homotopy-lift-3}
	\end{enumerate}
	Then, $\pi^{2(b-a+1)}\cdot g$ lies in the image of the map $\ho_{\mbf{K}(\scr{A})}(P^\bullet,M^\bullet)\to \ho_{\mbf{K}(\scr{A})}(P^\bullet,N^\bullet)$ induced by $f$.
\end{mylem}
\begin{proof}
	Let $C^\bullet$ be the cone of $f$, and let $\iota:N^\bullet\to C^\bullet$ be the canonical morphism. Applying the homological functor $\ho_{\dd(\scr{A})}(P^\bullet,-)$ to the distinguished triangle $\tau^{\leq a-1}C^\bullet\to C^\bullet\to\tau^{\geq a}C^\bullet\to (\tau^{\leq a-1}C^\bullet)[1]$ in the derived category $\dd(\scr{A})$, we obtain an exact sequence of $R$-modules (\cite[\href{https://stacks.math.columbia.edu/tag/0149}{0149}, \href{https://stacks.math.columbia.edu/tag/064B}{064B}]{stacks-project})
	\begin{align}
		\ho_{\mbf{K}(\scr{A})}(P^\bullet,\tau^{\leq a-1}C^\bullet)\to\ho_{\mbf{K}(\scr{A})}(P^\bullet,C^\bullet)\to\ho_{\mbf{K}(\scr{A})}(P^\bullet,\tau^{\geq a}C^\bullet).
	\end{align}
	The first term is zero by assumption (\ref{item:lem:pi-exact-homotopy-lift-2}), and $\tau^{\geq a}C^\bullet$ is $\pi^2$-exact by assumption (\ref{item:lem:pi-exact-homotopy-lift-3}). As $\tau^{\geq a}C^\bullet$ vanishes outside $[a,b]$ by assumption (\ref{item:lem:pi-exact-homotopy-lift-1}), the third term is $\pi^{2(b-a+1)}$-null by \ref{prop:pi-exact-homotopy}. We see that $\pi^{2(b-a+1)}\cdot(\iota\circ g)$ is zero in $\ho_{\mbf{K}(\scr{A})}(P^\bullet,C^\bullet)$. Therefore, the conclusion follows from the exact sequence (\cite[\href{https://stacks.math.columbia.edu/tag/0149}{0149}]{stacks-project})
	\begin{align}
		\ho_{\mbf{K}(\scr{A})}(P^\bullet,M^\bullet)\to\ho_{\mbf{K}(\scr{A})}(P^\bullet,N^\bullet)\to\ho_{\mbf{K}(\scr{A})}(P^\bullet,C^\bullet).
	\end{align}
\end{proof}

\section{Pseudo-coherence up to Bounded Torsion}\label{sec:ps-coh}
In this section, we fix integers $a\leq b$, a ring $R$ and an element $\pi$ of $R$. We remark that the universal bound $l$ that shall appear in each statement of this section depends only on the difference $b-a$ but not on $R$ or $\pi$.

\begin{mydefn}\label{defn:pi-pseudo-coh}
	Let $M^\bullet$ be a complex of $R$-modules.
	\begin{enumerate}
		\renewcommand{\labelenumi}{{\rm(\theenumi)}}
		\item A \emph{$\pi$-$[a,b]$-pseudo resolution} of $M^\bullet$ is a morphism $f:P^\bullet\to M^\bullet$ of complexes of $R$-modules, where $P^\bullet$ is a complex of finite free $R$-modules such that $P^i=0$ for any $i\notin [a,b]$, and where the map of cohomology groups $H^i(f):H^i(P^\bullet)\to H^i(M^\bullet)$ is a $\pi$-isomorphism for $i>a$ and $\pi$-surjective for $i=a$.
		\item We say that $M^\bullet$ is \emph{$\pi$-$[a,b]$-pseudo-coherent} if $M^i=0$ for any $i>b$ and if it admits a $\pi$-$[a,b]$-pseudo resolution. We say that an $R$-module $M$ is \emph{$\pi$-$[a,b]$-pseudo-coherent} if the complex $M[0]$ is $\pi$-$[a,b]$-pseudo-coherent.
	\end{enumerate}
\end{mydefn}

We follow the presentation of \cite[\href{https://stacks.math.columbia.edu/tag/064N}{064N}]{stacks-project} to establish some basic properties of this notion. The author does not know whether this notion is Zariski local on $R$ or not (cf. \cite[\href{https://stacks.math.columbia.edu/tag/066D}{066D}]{stacks-project}). This ad hoc notion only serves for the proof of our main theorem.

\begin{mylem}\label{lem:pi-pseudo-coh-translation}
	For any integers $a'\geq a$ and $b'\geq b$ with $a'\leq b'$, a $\pi$-$[a,b]$-pseudo-coherent complex of $R$-modules is also $\pi$-$[a',b']$-pseudo-coherent. 
\end{mylem}
\begin{proof}
	We only need to treat the case $a=a'$ and the case $b=b'$ separately. If $a=a'$, then it is clear that $M^i=0$ for any $i>b'$ and a $\pi$-$[a,b]$-pseudo resolution of $M^\bullet$ is also a $\pi$-$[a,b']$-pseudo resolution. If $b=b'$, then a $\pi$-$[a,b]$-pseudo resolution $P^\bullet\to M^\bullet$ induces a $\pi$-$[a',b]$-pseudo resolution $\sigma^{\geq a'}P^\bullet\to M^\bullet$.
\end{proof}

\begin{mylem}\label{lem:pi-pseudo-coh-iso}
	Let $M^\bullet$ and $N^\bullet$ be complexes of $R$-modules vanishing in degrees $>b$, and let $\alpha:M^\bullet\to N^\bullet$ be a morphism inducing a $\pi$-isomorphism on cohomology groups $H^i(\alpha):H^i(M^\bullet)\to H^i(N^\bullet)$ for any $i\geq a$.
	\begin{enumerate}
		\renewcommand{\labelenumi}{{\rm(\theenumi)}}
		\item If $M^\bullet$ is $\pi$-$[a,b]$-pseudo-coherent, then $N^\bullet$ is $\pi^2$-$[a,b]$-pseudo-coherent.\label{item:lem:pi-pseudo-coh-iso-1}
		\item If $N^\bullet$ is $\pi$-$[a,b]$-pseudo-coherent, then $M^\bullet$ is $\pi^l$-$[a,b]$-pseudo-coherent for an integer $l\geq 0$ depending only on $b-a$.\label{item:lem:pi-pseudo-coh-iso-2}
	\end{enumerate}
\end{mylem}
\begin{proof}
	(\ref{item:lem:pi-pseudo-coh-iso-1}) We take a $\pi$-$[a,b]$-pseudo resolution $f:P^\bullet\to M^\bullet$. In particular, $H^i(f)$ is a $\pi$-isomorphism for any $i>a$ and $\pi$-surjective for $i=a$. Hence, $H^i(\alpha\circ f)=H^i(\alpha)\circ H^i(f)$ is $\pi^2$-isomorphism for any $i>a$ and $\pi^2$-surjective for $i=a$, which shows that $\alpha\circ f:P^\bullet\to N^\bullet$ is a $\pi^2$-$[a,b]$-pseudo resolution.
	
	(\ref{item:lem:pi-pseudo-coh-iso-2}) Let $g:P^\bullet\to N^\bullet$ be a $\pi$-$[a,b]$-pseudo resolution. We obtain from \ref{lem:pi-exact-homotopy-lift} a morphism $f:P^\bullet\to M^\bullet$ lifting $\pi^l\cdot g$ up to homotopy for $l=2(b-a+1)$. Thus, for any $i\in\bb{Z}$, we have
	\begin{align}
		H^i(\pi^l\cdot g)=H^i(\alpha)\circ H^i(f).
	\end{align}
	Notice that $H^i(\pi^l\cdot g)$ is a $\pi^{l+1}$-isomorphism for $i>a$ and $\pi^{l+1}$-surjective for $i=a$, and that $H^i(\alpha)$ is a $\pi$-isomorphism for $i\geq a$. We see that $H^i(f)$ is a $\pi^{l+2}$-isomorphism for $i>a$ and $\pi^{l+2}$-surjective for $i=a$. Thus, $f:P^\bullet\to M^\bullet$ is a $\pi^{l+2}$-$[a,b]$-pseudo resolution.
\end{proof}

\begin{myprop}\label{prop:pi-pseudo-coh-derived}
	Let $M^\bullet$ and $N^\bullet$ be two complexes of $R$-modules vanishing in degree $>b$. Assume that they are isomorphic in the derived category $\dd(R)$. Then, if $M^\bullet$ is $\pi$-$[a,b]$-pseudo-coherent, then $N^\bullet$ is $\pi^l$-$[a,b]$-pseudo-coherent for an integer $l\geq 0$ depending only on $b-a$.
\end{myprop}
\begin{proof}
	Let $P^\bullet\to N^\bullet$ be a bounded above projective resolution with the same top degree. The assumption implies that there is a quasi-isomorphism of complexes $P^\bullet\to M^\bullet$ (\cite[\href{https://stacks.math.columbia.edu/tag/064B}{064B}]{stacks-project}). The conclusion follows from applying \ref{lem:pi-pseudo-coh-iso} to $P^\bullet\to M^\bullet$ and $P^\bullet\to N^\bullet$.
\end{proof}

\begin{mylem}\label{lem:pi-pseudo-coh-lift}
	Let $\alpha:M_1^\bullet\to M_2^\bullet$ be a morphism of $\pi$-$[a,b]$-pseudo-coherent complexes of $R$-modules. Given $\pi$-$[a,b]$-pseudo resolutions $f_i:P_i^\bullet\to M_i^\bullet$ ($i=1,2$), there exists a morphism of complexes $\alpha':P_1^\bullet\to P_2^\bullet$ such that $(\pi^l\cdot \alpha)\circ f_1$ is homotopic to $f_2\circ \alpha'$ for an integer $l\geq 0$ depending only on $b-a$.
\end{mylem}
\begin{proof}
	It follows directly from \ref{lem:pi-exact-homotopy-lift}.
\end{proof}

\begin{mylem}\label{lem:pi-pseudo-coh-triangle}
	Let $M_1^\bullet\stackrel{\alpha}{\longrightarrow} M_2^\bullet\stackrel{\beta}{\longrightarrow} M_3^\bullet\stackrel{\gamma}{\longrightarrow}M_1^\bullet[1]$ be a distinguished triangle in the homotopy category $\mbf{K}(R)$. Assume that $M_1^\bullet$ is $\pi$-$[a+1,b+1]$-pseudo-coherent, $M_2^\bullet$ is $\pi$-$[a,b]$-pseudo-coherent, and $M_3^i=0$ for any $i>b$. Then, $M_3^\bullet$ is $\pi^l$-$[a,b]$-pseudo-coherent for an integer $l\geq 0$ depending only on $b-a$.
\end{mylem}
\begin{proof}
	We take a $\pi$-$[a+1,b+1]$-pseudo resolution $f_1:P_1^\bullet\to M_1^\bullet$ and a $\pi$-$[a,b]$-pseudo resolution $f_2:P_2^\bullet\to M_2^\bullet$. By \ref{lem:pi-pseudo-coh-lift}, there exists a morphism $\alpha':P_1^\bullet\to P_2^\bullet$ lifting $\pi^l\cdot \alpha$ in $\mbf{K}(R)$ for an integer $l\geq 0$ depending only on $b-a$. If we denote its cone by $P_3^\bullet$, then we have a morphism of distinguished triangles in $\mbf{K}(R)$.
	\begin{align}
		\xymatrix{
			P_1^\bullet\ar[r]^-{\alpha'}\ar[d]_-{\pi^l\cdot f_1}& P_2^\bullet\ar[r]^-{\beta'}\ar[d]_-{f_2}& P_3^\bullet\ar[r]^-{\gamma'}\ar[d]_-{f_3}&P_1^\bullet[1]\ar[d]^-{\pi^l\cdot f_1[1]}\\
			M_1^\bullet\ar[r]^-{\alpha}& M_2^\bullet\ar[r]^-{\beta}& M_3^\bullet\ar[r]^-{\gamma}&M_1^\bullet[1]
		}
	\end{align}
	Let $C_1^\bullet$, $C_2^\bullet$, $C_3^\bullet$ be the cones of $\pi^l\cdot f_1$, $f_2$, $f_3$ respectively. We obtain a distinguished triangle $C_1^\bullet\to C_2^\bullet\to C_3^\bullet\to C_1^\bullet[1]$ in  $\mbf{K}(R)$ (\cite[\href{https://stacks.math.columbia.edu/tag/05R0}{05R0}]{stacks-project}). By assumption, $\tau^{\geq (a+1)}C_1^\bullet$ is $\pi^{2(l+1)}$-exact and $\tau^{\geq a}C_2^\bullet$ is $\pi^2$-exact. Thus, we see that $\tau^{\geq a}C_3^\bullet$ is $\pi^{2(l+2)}$-exact. As $P_3^\bullet$ vanishes outside $[a,b]$, $P_3^\bullet\to M_3^\bullet$ is a $\pi^{2(l+2)}$-$[a,b]$-pseudo resolution.
\end{proof}

\begin{myprop}\label{prop:pi-pseudo-coh-ext}
	Let $0\longrightarrow M_1^\bullet\stackrel{\alpha}{\longrightarrow} M_2^\bullet\stackrel{\beta}{\longrightarrow} M_3^\bullet\longrightarrow 0$ be an exact sequence of complexes of $R$-modules.
	\begin{enumerate}
		\renewcommand{\labelenumi}{{\rm(\theenumi)}}
		\item Assume that $M_1^\bullet$ is $\pi$-$[a+1,b+1]$-pseudo-coherent and $M_2^\bullet$ is $\pi$-$[a,b]$-pseudo-coherent. Then, $M_3^\bullet$ is $\pi^l$-$[a,b]$-pseudo-coherent for an integer $l\geq 0$ depending only on $b-a$.\label{item:prop:pi-pseudo-coh-ext-1}
		\item Assume that $M_1^\bullet$ and $M_3^\bullet$ are $\pi$-$[a,b]$-pseudo-coherent. Then, $M_2^\bullet$ is $\pi^l$-$[a,b]$-pseudo-coherent for an integer $l\geq 0$ depending only on $b-a$.\label{item:prop:pi-pseudo-coh-ext-2}
		\item Assume that $M_2^\bullet$ is $\pi$-$[a-1,b-1]$-pseudo-coherent and $M_3^\bullet$ is $\pi$-$[a-2,b-1]$-pseudo-coherent. Then, $M_1^\bullet$ is $\pi^l$-$[a-1,b]$-pseudo-coherent for an integer $l\geq 0$ depending only on $b-a$.\label{item:prop:pi-pseudo-coh-ext-3}
	\end{enumerate}
\end{myprop}
\begin{proof}
	Let $C^\bullet$ be the cone of $\alpha:M_1^\bullet\to M_2^\bullet$. Then, the natural morphism $C^\bullet \to M_3^\bullet$ is a quasi-isomorphism. 
	
	(\ref{item:prop:pi-pseudo-coh-ext-1}) In this case, $M_3^i=0$ and $C^i=0$ for any $i>b$, and actually $C^\bullet$ is $\pi^l$-$[a,b]$-pseudo-coherent for an integer $l\geq 0$ depending only on $b-a$ by applying \ref{lem:pi-pseudo-coh-triangle} to the distinguished triangle $M_1^\bullet\to M_2^\bullet\to C^\bullet \to M_1^\bullet[1]$ in $\mbf{K}(R)$. Thus, $M_3^\bullet$ is $\pi^{2l}$-$[a,b]$-pseudo-coherent by \ref{lem:pi-pseudo-coh-iso}.(\ref{item:lem:pi-pseudo-coh-iso-1}). 
	
	(\ref{item:prop:pi-pseudo-coh-ext-2}) In this case, $M_2^i=0$ and $C^i=0$ for any $i>b$, and $C^\bullet$ is $\pi^l$-$[a,b]$-pseudo-coherent for an integer $l\geq 0$ depending only on $b-a$ by \ref{lem:pi-pseudo-coh-iso}.(\ref{item:lem:pi-pseudo-coh-iso-2}). Thus, $M_2^\bullet$ is $\pi^l$-$[a,b]$-pseudo-coherent for an integer $l\geq 0$ depending only on $b-a$ by applying \ref{lem:pi-pseudo-coh-triangle} to the distinguished triangle $C^\bullet[-1]\to M_1^\bullet\to M_2^\bullet\to C^\bullet$ in $\mbf{K}(R)$. 
	
	(\ref{item:prop:pi-pseudo-coh-ext-3}) In this case, $M_1^i=0$ and $C^i=0$ for any $i>b-1$, and $C^\bullet$ is $\pi^l$-$[a-2,b-1]$-pseudo-coherent for an integer $l\geq 0$ depending only on $b-a$ by \ref{lem:pi-pseudo-coh-iso}.(\ref{item:lem:pi-pseudo-coh-iso-2}). Thus, $M_1^\bullet$ is $\pi^l$-$[a-1,b]$-pseudo-coherent for an integer $l\geq 0$ depending only on $b-a$ by applying \ref{lem:pi-pseudo-coh-triangle} to the distinguished triangle $M_2^\bullet[-1]\to C^\bullet[-1]\to M_1^\bullet\to M_2^\bullet$ in $\mbf{K}(R)$.
\end{proof}

\begin{mycor}\label{cor:pi-pseudo-coh-coh}
	Let $M^\bullet$ be a complex of $R$-modules vanishing in degrees $>b$. Assume that the cohomology group $H^i(M^\bullet)$ is $\pi$-$[a-i,b-i]$-pseudo-coherent for any $i\in[a,b]$. Then, $M^\bullet$ is $\pi^l$-$[a,b]$-pseudo-coherent for an integer $l\geq 0$ depending only on $b-a$.
\end{mycor}
\begin{proof}
	We proceed by induction on $b-a$. If $a=b$, then $\tau^{\geq a}M^\bullet=H^a(M^\bullet)[-a]$ is $\pi$-$[a,b]$-pseudo-coherent by assumption. Thus, $M^\bullet$ is $\pi^l$-$[a,b]$-pseudo-coherent for an integer $l\geq 0$ depending only on $b-a$ by \ref{lem:pi-pseudo-coh-iso}.(\ref{item:lem:pi-pseudo-coh-iso-2}). In general, consider the exact sequence of complexes of $R$-modules
	\begin{align}
		&0\longrightarrow\tau^{\leq (b-1)}M^\bullet\longrightarrow M^\bullet\longrightarrow (M^{b-1}/\ke(\df^{b-1})[1-b]\to M^b[-b])\longrightarrow 0.
	\end{align}
	As the natural morphism of complexes $N^\bullet=(M^{b-1}/\ke(\df^{b-1})[1-b]\to M^b[-b])\to H^b(M^\bullet)[-b]$ is a quasi-isomorphism, $N^\bullet$ is $\pi^l$-$[a,b]$-pseudo-coherent for an integer $l\geq 0$ depending only on $b-a$ by \ref{lem:pi-pseudo-coh-iso}.(\ref{item:lem:pi-pseudo-coh-iso-2}). Notice that $\tau^{\leq (b-1)}M^\bullet$ is $\pi^l$-$[a,b]$-pseudo-coherent for an integer $l\geq 0$ depending only on $b-a-1$ by induction. The conclusion follows from \ref{prop:pi-pseudo-coh-ext}.(\ref{item:prop:pi-pseudo-coh-ext-2}).
\end{proof}

\begin{mylem}\label{lem:pi-null-bc}
	Let $M^\bullet$ be a $\pi$-exact complex of $R$-modules vanishing outside $[a,b]$, and let $N^\bullet$ be a complex of $R$-modules. Then, $M^\bullet\otimes_R^{\dl} N^\bullet$ is $\pi^l$-exact (i.e., any complex representing the derived tensor product $M^\bullet\otimes_R^{\dl} N^\bullet$ is $\pi^l$-exact) for an integer $l\geq 0$ depending only on $b-a$.
\end{mylem}
\begin{proof}
	We proceed by induction on $b-a$. If $a=b$, then the multiplication by $\pi$ on $M^\bullet=H^a(M^\bullet)[-a]$ factors through zero. Hence, $M^\bullet\otimes_R^{\dl} N^\bullet$ is $\pi$-exact. In general, consider the distinguished triangle in the derived category $\dd(R)$,
	\begin{align}
		(\tau^{\leq (b-1)}M^\bullet)\otimes_R^{\dl} N^\bullet\longrightarrow M^\bullet\otimes_R^{\dl} N^\bullet\longrightarrow H^b(M^\bullet)[-b]\otimes_R^{\dl} N^\bullet\longrightarrow(\tau^{\leq (b-1)}M^\bullet)\otimes_R^{\dl} N^\bullet[1].
	\end{align}
	Notice that $(\tau^{\leq (b-1)}M^\bullet)\otimes_R^{\dl} N^\bullet$ is $\pi^l$-exact for an integer $l\geq 0$ depending only on $b-a$ by induction. By the long exact sequence of cohomology groups, we see that $M^\bullet\otimes_R^{\dl} N^\bullet$ is $\pi^l$-exact for an integer $l\geq 0$ depending only on $b-a$.
\end{proof}

\begin{myprop}\label{prop:pi-pseudo-coh-bc}
	Let $M^\bullet$ be a $\pi$-$[a,b]$-pseudo-coherent complex of $R$-modules, and let $S$ be an $R$-algebra. Then, $\tau^{\geq a}(S\otimes_R^{\dl}M^\bullet)$ is represented by a $\pi^l$-$[a,b]$-pseudo-coherent complex of $S$-modules for an integer $l\geq 0$ depending only on $b-a$.
\end{myprop}
\begin{proof}
	We take a bounded above flat resolution $F^\bullet\to M^\bullet$ with the same top degree. By \ref{lem:pi-pseudo-coh-iso}.(\ref{item:lem:pi-pseudo-coh-iso-2}), $\sigma^{\geq a-1}F^\bullet$ is a $\pi^l$-$[a,b]$-pseudo-coherent complex of flat $R$-modules for an integer $l\geq 0$ depending only on $b-a$. Let $P^\bullet\to \sigma^{\geq a-1}F^\bullet$ be a $\pi^l$-$[a,b]$-pseudo resolution, and let $C^\bullet$ be its cone. Consider the distinguished triangle in $\mbf{K}(S)$,
	\begin{align}\label{eq:prop:pi-pseudo-coh-bc}
		S\otimes_R P^\bullet\longrightarrow S\otimes_R \sigma^{\geq a-1}F^\bullet\longrightarrow S\otimes_R C^\bullet\longrightarrow S\otimes_R P^\bullet[1].
	\end{align}
	 Notice that $\tau^{\geq a}C^\bullet$ is a $\pi^{2l}$-exact complex vanishing outside $[a,b]$ and that $S\otimes_R C^\bullet\cong S\otimes_R^{\dl} C^\bullet$ in $\dd(S)$ by construction. After enlarging $l$ by \ref{lem:pi-null-bc}, we may assume that $\tau^{\geq a}(S\otimes_R C^\bullet)=\tau^{\geq a}(S\otimes_R^{\dl}\tau^{\geq a}C^\bullet)$ is $\pi^l$-exact. By the long exact sequence associated to \eqref{eq:prop:pi-pseudo-coh-bc}, we see that $S\otimes_R P^\bullet\to S\otimes_R \sigma^{\geq a-1}F^\bullet$ is a $\pi^{2l}$-$[a,b]$-pseudo resolution of complexes of $S$-modules, and thus so is the composition
	 \begin{align}
	 	S\otimes_R P^\bullet\longrightarrow S\otimes_R \sigma^{\geq a-1}F^\bullet\longrightarrow\tau^{\geq a}(S\otimes_R \sigma^{\geq a-1}F^\bullet)=\tau^{\geq a}(S\otimes_R F^\bullet),
	 \end{align}
	 where the target is a complex of $S$-modules representing $\tau^{\geq a}(S\otimes_R^{\dl}M^\bullet)$ and vanishing in degrees $>b$.
\end{proof}

\begin{mydefn}\label{defn:pi-fini}
	Let $M$ be an $R$-module. We say that $M$ is \emph{of $\pi$-finite type} if there exists $n\in\bb{N}$ and a $\pi$-surjective $R$-linear homomorphism $R^{\oplus n}\to M$.
\end{mydefn}
This definition is a special case of \ref{defn:almost} below.

\begin{mylem}\label{lem:pi-pseudo-coh-noeth}
	Assume that $R$ is Noetherian. Let $M$ be an $R$-module.
	\begin{enumerate}
		\renewcommand{\labelenumi}{{\rm(\theenumi)}}
		\item If $M$ is of $\pi$-finite type, then it is $\pi$-$[a,b]$-pseudo-coherent for any integers $a\leq0\leq b$. Conversely, if $M$ is $\pi$-$[a,b]$-pseudo-coherent for some integers $a\leq 0\leq b$, then $M$ is of $\pi$-finite type.\label{item:lem:pi-pseudo-coh-noeth-1}
		\item If $M$ is of $\pi$-finite type, then so are its subquotients. Conversely, if $M$ admits a finite filtration of length $l$ (\cite[\href{https://stacks.math.columbia.edu/tag/0121}{0121}]{stacks-project}) whose graded pieces are of $\pi$-finite type, then $M$ is of $\pi^l$-finite type. \label{item:lem:pi-pseudo-coh-noeth-2}
	\end{enumerate}
\end{mylem}
\begin{proof}
	(\ref{item:lem:pi-pseudo-coh-noeth-1}) If $M$ is of $\pi$-finite type, then there is a finitely generated $R$-submodule $N$ of $M$ such that $\pi M\subseteq N$. Since $R$ is Noetherian, $N$ is pseudo-coherent (\cite[\href{https://stacks.math.columbia.edu/tag/066E}{066E}]{stacks-project}). Hence, $M$ is $\pi$-$[a,b]$-pseudo-coherent for any integers $a\leq0\leq b$. Conversely, if $M$ is $\pi$-$[a,b]$-pseudo-coherent, we take a $\pi$-$[a,b]$-pseudo resolution $P^\bullet \to M[0]$. As a subquotient of a finitely generated $R$-module, $H^0(P^\bullet)$ is also finitely generated as $R$ is Noetherian. Hence, $M=H^0(M[0])$ is of $\pi$-finite type.
	
	(\ref{item:lem:pi-pseudo-coh-noeth-2}) Let $N$ be a finitely generated $R$-submodule of $M$ such that $\pi M\subseteq N$. Let $M_0\subseteq M_1$ be two $R$-submodules of $M$. Notice that $N\cap M_1$ is a finitely generated $R$-module as $R$ is Noetherian. The conclusion follows from the $\pi$-surjectivity of $N\cap M_1\to M_1/M_0$. Conversely, assume that there is a finite filtration $0=M_0\subseteq M_1\subseteq \cdots \subseteq M_l=M$ such that $M_{i+1}/M_i$ is of $\pi$-finite type. Then, we see that $M$ is of $\pi^l$-finite type by inductively using \cite[2.7.14.(\luoma{2})]{abbes2020suite}.
\end{proof}

\begin{myprop}\label{prop:pi-pseudo-coh-noeth}
	Assume that $R$ is Noetherian. Let $M^\bullet$ be a complex of $R$-modules.
	\begin{enumerate}
		\renewcommand{\labelenumi}{{\rm(\theenumi)}}
		\item If $H^i(M^\bullet)$ is of $\pi$-finite type for any $i\geq a$ and if $M^i=0$ for any $i>b$, then $M^\bullet$ is $\pi^l$-$[a,b]$-pseudo-coherent for an integer $l\geq 0$ depending only on $b-a$.\label{item:prop:pi-pseudo-coh-noeth-1}
		\item If $M^\bullet$ is $\pi$-$[a,b]$-pseudo-coherent, then $H^i(M^\bullet)$ is of $\pi$-finite type for any $i\geq a$.\label{item:prop:pi-pseudo-coh-noeth-2}
	\end{enumerate}
\end{myprop}
\begin{proof}
	(\ref{item:prop:pi-pseudo-coh-noeth-1}) follows from \ref{lem:pi-pseudo-coh-noeth}.(\ref{item:lem:pi-pseudo-coh-noeth-1}) and \ref{cor:pi-pseudo-coh-coh}; and (\ref{item:prop:pi-pseudo-coh-noeth-2}) follows from the same argument of \ref{lem:pi-pseudo-coh-noeth}.(\ref{item:lem:pi-pseudo-coh-noeth-1}).
\end{proof}

\section{Glueing Sheaves up to Bounded Torsion}\label{sec:glue}
In this section, we fix a ring $R$ and an element $\pi$ of $R$.

\begin{mypara}\label{para:pi-cartesian}
	Let $E/C$ be a fibred site, and let $\ca{O}=(\ca{O}_\alpha)_{\alpha\in\ob(C)}$ be a sheaf of $R$-algebras over the total site $E$ (\cite[\Luoma{6}.7.4.1]{sga4-2}). We say that an $\ca{O}$-module $\ca{F}=(\ca{F}_\alpha)_{\alpha\in\ob(C)}$ on $E$ is \emph{$\pi$-Cartesian}, if for every morphism $f:\beta\to \alpha$ in $C$, the induced map $f^*\ca{F}_\alpha\to\ca{F}_\beta$ is a $\pi$-isomorphism of $\ca{O}_\beta$-modules.
\end{mypara}

\begin{mypara}
	Let $E$ be a category. Recall that a semi-representable object of $E$ is a family $\{U_i\}_{i\in I}$ of objects of $E$. A morphism $\{U_i\}_{i\in I}\to \{V_j\}_{j\in J}$ of semi-representable objects of $E$ is given by a map $\alpha:I\to J$ and for every $i\in I$ a morphism $f_i:U_i\to V_{\alpha(i)}$ (\cite[\href{https://stacks.math.columbia.edu/tag/01G0}{01G0}]{stacks-project}). Assume that $E$ is a site (\cite[\Luoma{2}.1.1.5]{sga4-1}) where fibred products are representable. For a semi-representable object $K=\{U_i\}_{i\in I}$ of objects of $E$, let $E_{/K}=\coprod_{i\in I}E_{/U_i}$ be the disjoint union of the localizations of $E$ at $U_i$ (\cite[\href{https://stacks.math.columbia.edu/tag/09WK}{09WK}]{stacks-project}). We note that for any morphism $K'\to K$ of semi-representable objects of $E$, the canonical morphism of sites $E_{/K'}\to E_{/K}$ induced by the cocontinuous forgetful functor $E_{/K'}\to E_{/K}$ is also induced by the continuous base change functor $E_{/K}\to E_{/K'}$ (\cite[\href{https://stacks.math.columbia.edu/tag/0D85}{0D85}, \href{https://stacks.math.columbia.edu/tag/0D87}{0D87}]{stacks-project}).
	
	Let $r\in\bb{N}\cup\{\infty\}$. For an $r$-truncated simplicial semi-representable object $K_\bullet=(K_n)_{[n]\in\ob(\Delta_{\leq r})}$ of $E$ (where each $K_n$ is a semi-representable object of $E$), we denote by $E_{/K_\bullet}$ the fibred site over the $r$-truncated simplicial category $\Delta_{\leq r}$ whose fibre over $[n]$ is $E_{/K_n}$ (\cite[\href{https://stacks.math.columbia.edu/tag/0D8A}{0D8A}]{stacks-project}). We denote by $\nu:E_{/K_\bullet}\to E$ the augmentation, and by $\nu_n:E_{/K_n}\to E$ the corresponding morphism of sites for any $n\in\bb{N}_{\leq r}$ (\cite[\href{https://stacks.math.columbia.edu/tag/0D8B}{0D8B}]{stacks-project}).
\end{mypara}

\begin{mylem}\label{lem:base-change-mor}
	Let $E$ be a site where fibred products are representable, let
	\begin{align}\label{diam:base-change-mor-data}
		\xymatrix{
			Y'\ar[r]^-{g'}\ar[d]_-{f'}&X'\ar[d]^-{f}\\
			Y\ar[r]^-{g}&X
		}
	\end{align}
	be a commutative diagram in $E$, and let $\ca{F}$ a sheaf on $E_{/X'}$. Then, the following diagram is commutative
	\begin{align}\label{diam:base-change-mor}
		\xymatrix{
			f'_*f'^*g^*f_*\ca{F}\ar[r]^-{\sim}&f'_*g'^*f^*f_*\ca{F}\ar[d]^-{f'_*g'^*\beta_f|_{\ca{F}}}\\
			g^*f_*\ca{F}\ar[u]^-{\alpha_{f'}|_{g^*f_*\ca{F}}}\ar[d]_-{g^*f_*\alpha_{g'}|_{\ca{F}}}&f'_*g'^*\ca{F}\\
			g^*f_*g'_*g'^*\ca{F}\ar[r]^-{\sim}&g^*g_*f'_*g'^*\ca{F}\ar[u]_-{\beta_g|_{f'_*g'^*\ca{F}}}
		}
	\end{align}
	where $\alpha_{f'}$ (resp. $\alpha_{g'}$) is the adjunction morphism $\id\to f'_*f'^*$ (resp. $\id \to g'_*g'^*$), $\beta_f$ (resp. $\beta_g$) is the adjunction morphism $f^*f_*\to \id$ (resp. $g^*g_*\to \id$), and the horizontal isomorphisms are induced by the canonical isomorphisms $f'^*g^*\iso (g\circ f')^*=(f\circ g')^*\stackrel{\sim}{\longleftarrow} g'^*f^*$ and $f_*g'_*\iso (f\circ g')_*=(g\circ f')_*\stackrel{\sim}{\longleftarrow} g_*f'_*$. We call the morphism 
	\begin{align}\label{eq:base-change-mor}
		g^*f_*\ca{F}\longrightarrow f'_*g'^*\ca{F}
	\end{align}
	defined by the composition in either upper or lower way of \eqref{diam:base-change-mor} the \emph{base change morphism}.
	
	Moreover, if the diagram \eqref{diam:base-change-mor-data} is Cartesian, then the base change morphism $g^*f_*\ca{F}\to f'_*g'^*\ca{F}$ is an isomorphism.
\end{mylem}
\begin{proof}
	As $(E_{/U})_{U\in \ob(E)}$ forms a fibred site over $E$ (\cite[\Luoma{6}.7.3.2]{sga4-2}), the commutativity of \eqref{diam:base-change-mor} is a special case of \cite[\Luoma{17}.2.1.3]{sga4-3} applied to the asssociated fibred topos. One can also check it by unwinding the definitions. Indeed, for any object $V$ of $E_{/Y}$, we have
	\begin{align}
		g^*f_*\ca{F}(V/Y)&=f_*\ca{F}(V/X)=\ca{F}(V\times_XX'/X'),\\
		f'_*g'^*\ca{F}(V/Y)&=g'^*\ca{F}(V\times_YY'/Y')=\ca{F}(V\times_YY'/X').
	\end{align}
	One can check by definitions that the composition $g^*f_*\ca{F}(V/Y)\to f'_*g'^*\ca{F}(V/Y)$ in either upper or lower way of \eqref{diam:base-change-mor} coincides with the restriction map of $\ca{F}$ along the canonical morphism $V\times_YY'\to V\times_XX'$ over $X'$. Moreover, if the diagram \eqref{diam:base-change-mor-data} is Cartesian, then $V\times_YY'\to V\times_XX'$ is an isomorphism so that $g^*f_*\ca{F}(V/Y)\to f'_*g'^*\ca{F}(V/Y)$ is an isomorphism.
\end{proof}

\begin{myprop}\label{prop:pi-sheaf-glue}
	Let $E$ be a site where fibred products are representable, let $\ca{O}$ be a sheaf of $R$-algebras on $E$, and let $\{U_i\to X\}_{i\in I}$ be a covering in $E$. Consider the $2$-truncated \v Cech hypercovering (\cite[\href{https://stacks.math.columbia.edu/tag/01G6}{01G6}]{stacks-project})
	\begin{align}
		\xymatrix{
			K_\bullet=(\{U_i\times_XU_j\times_XU_k\}_{i,j,k\in I}\ar[r]\ar@<1.2ex>[r]\ar@<-1.2ex>[r]& \{U_i\times_XU_j\}_{i,j\in I}\ar@<0.6ex>[r]\ar@<-0.6ex>[r]\ar@<0.6ex>[l]\ar@<-0.6ex>[l]& \{U_i\}_{i\in I}\ar[l]),
		}
	\end{align}
	regarded as a $2$-truncated simplicial semi-representable object of $E_{/X}$. Let $\ca{F}_\bullet=(\ca{F}_n)_{[n]\in\Delta_{\leq 2}}$ be a $\pi$-Cartesian $\ca{O}_{/K_\bullet}$-module over the $2$-truncated simplicial ringed site $E_{/K_\bullet}$, and we put $\ca{F}=\nu_*\ca{F}_\bullet$ where $\nu:E_{/K_\bullet}\to E_{/X}$ is the augmentation. Then, the canonical map $\nu_0^*\ca{F}\to \ca{F}_0$ is a $\pi^8$-isomorphism.
\end{myprop}
\begin{proof}
	For any $i,j,k\in I$, we denote by $f_i:U_i\to X$, $f_{ij}:U_i\times_XU_j\to X$, $f_{ijk}:U_i\times_XU_j\times_XU_k\to X$ the canonical morphisms, and denote by $\ca{G}_i$, $\ca{G}_{ij}$, $\ca{G}_{ijk}$ the restrictions of $\ca{F}_0$, $\ca{F}_1$, $\ca{F}_2$ to $U_i$, $U_i\times_XU_j$, $U_i\times_XU_j\times_XU_k$ respectively. By definition (\cite[\href{https://stacks.math.columbia.edu/tag/09WM}{09WM}]{stacks-project}), we have
	\begin{align}\label{eq:prop:pi-sheaf-glue-1}
		\ca{F}=\mrm{Eq}(\prod_{j\in I}f_{j*}\ca{G}_j\rightrightarrows\prod_{j,k\in I}f_{jk*}\ca{G}_{jk}).
	\end{align}
	We need to show that the canonical map (note that the restriction functor $f_i^*$ of sheaves commute with any limits as it admits a left adjoint $f_{i!}$),
	\begin{align}\label{eq:lem:pi-sheaf-glue-1}
		f_i^*\ca{F}=\mrm{Eq}(\prod_{j\in I}f_i^*f_{j*}\ca{G}_j\rightrightarrows\prod_{j,k\in I}f_i^*f_{jk*}\ca{G}_{jk})\longrightarrow \ca{G}_i
	\end{align}
	given by composing the projection on the $i$-th component with the adjunction morphism $f_i^*f_{i*}\ca{G}_i\to \ca{G}_i$, is a $\pi^8$-isomorphism for any $i\in I$. Fixing $i\in I$, for any $j,k\in I$, we name some natural arrows as indicated in the following commutative diagram.
	\begin{align}\label{diam:lem:pi-sheaf-glue-1}
		\xymatrix{
			U_i\times_XU_j\times_XU_k\ar[r]^-{h_{jk}}\ar@/_2pc/[dd]_-{g_{jk}}\ar[d]_-{\beta_k}&U_j\times_XU_k\ar@/^2pc/[dd]^-{f_{jk}}\ar[d]^-{\alpha_k}\\
			U_i\times_XU_j\ar[r]^-{h_j}\ar[d]_-{g_j}&U_j\ar[d]^-{f_j}\\
			U_i\ar[r]^-{f_i}&X
		}
	\end{align}
	Thus, we have a canonical commutative diagram of sheaves on $E_{/U_i}$,
	\begin{align}\label{diam:lem:pi-sheaf-glue-2}
		\xymatrix{
			g_{j*}g_j^*\ca{G}_i\ar[d]\ar[r]&g_{j*}\ca{G}_{ij}\ar[d]&g_{j*}h_j^*\ca{G}_j\ar[d]\ar[l]&f_i^*f_{j*}\ca{G}_j\ar[d]\ar[l]_-{\sim}\\
			g_{jk*}g_{jk}^*\ca{G}_i\ar[r]&g_{jk*}\ca{G}_{ijk}&g_{jk*}h_{jk}^*\ca{G}_{jk}\ar[l]&f_i^*f_{jk*}\ca{G}_{jk}\ar[l]_-{\sim}\\
			g_{k*}g_k^*\ca{G}_i\ar[u]\ar[r]&g_{k*}\ca{G}_{ik}\ar[u]&g_{k*}h_k^*\ca{G}_k\ar[u]\ar[l]& f_i^*f_{k*}\ca{G}_k\ar[u]\ar[l]_-{\sim}
		}
	\end{align}
	obtained by the following steps: 
	\begin{enumerate}
		\renewcommand{\labelenumi}{{\rm(\theenumi)}}
		\item The structure of $\ca{F}_\bullet$ gives a canonical commutative diagram
		\begin{align}\label{diam:lem:pi-sheaf-glue-2-1}
			\xymatrix{
				\beta_k^*g_j^*\ca{G}_i\ar[r]\ar@{=}[d]&\beta_k^*\ca{G}_{ij}\ar[d]&\beta_k^*h_j^*\ca{G}_j\ar[d]\ar[l]&\beta_k^*h_j^*f_j^*f_{j*}\ca{G}_j\ar[l]\ar@{=}[r]&h_{jk}^*\alpha_k^*f_j^*f_{j*}\ca{G}_j\ar[d]\\
				g_{jk}^*\ca{G}_i\ar[r]&\ca{G}_{ijk}&h_{jk}^*\ca{G}_{jk}\ar[l]&h_{jk}^*f_{jk}^*f_{jk*}\ca{G}_{jk}\ar@{=}[r]\ar[l]&h_{jk}^*\alpha_k^*f_j^*f_{j*}\alpha_{k*}\ca{G}_{jk}
			}
		\end{align}
		where the two horizontal arrows on the right square are induced by the adjunction morphisms $f_j^*f_{j*}\to \id$ and $f_{jk}^*f_{jk*}\to\id$. The right square is commutative since the adjunction morphism $f_{jk}^*f_{jk*}\to\id$ is the composition of $f_j^*f_{j*}\to \id$ with $\alpha_k^*\alpha_{k*}\to \id$.
		\item Applying $\beta_{k*}$ to \eqref{diam:lem:pi-sheaf-glue-2-1} and composing with the adjunction morphism $\id\to \beta_{k*}\beta_k^*$, we obtain a canonical commutative diagram
		\begin{align}\label{diam:lem:pi-sheaf-glue-2-2}
			\xymatrix{
				g_j^*\ca{G}_i\ar[r]\ar[d]&\ca{G}_{ij}\ar[d]&h_j^*\ca{G}_j\ar[d]\ar[l]&h_j^*f_j^*f_{j*}\ca{G}_j\ar[l]\ar[d]\\
				\beta_{k*}g_{jk}^*\ca{G}_i\ar[r]&\beta_{k*}\ca{G}_{ijk}&\beta_{k*}h_{jk}^*\ca{G}_{jk}\ar[l]&h_j^*f_j^*f_{jk*}\ca{G}_{jk}\ar[l]
			}
		\end{align}
		where we used the fact that the vertical arrow on the right of \eqref{diam:lem:pi-sheaf-glue-2-1} is the the image of the vertical arrow on the right of \eqref{diam:lem:pi-sheaf-glue-2-2} under $\beta_k^*$.
		\item Applying $g_{j*}$ to \eqref{diam:lem:pi-sheaf-glue-2-2} and composing with the adjunction morphisms $\id\to g_{j*}g_j^*$ on the right, we obtain a canonical commutative diagram
		\begin{align}\label{diam:lem:pi-sheaf-glue-2-3}
			\xymatrix{
				g_{j*}g_j^*\ca{G}_i\ar[r]\ar[d]&g_{j*}\ca{G}_{ij}\ar[d]&g_{j*}h_j^*\ca{G}_j\ar[d]\ar[l]&f_i^*f_{j*}\ca{G}_j\ar[l]\ar[d]\\
				g_{jk*}g_{jk}^*\ca{G}_i\ar[r]&g_{jk*}\ca{G}_{ijk}&g_{jk*}h_{jk}^*\ca{G}_{jk}\ar[l]&f_i^*f_{jk*}\ca{G}_{jk}\ar[l]
			}
		\end{align}
		\item Going through the construction and using the fact that the adjunction morphism $\id\to g_{jk*}g_{jk}^*$ is the composition of $\id\to \beta_{k*}\beta_k^*$ with $\id\to g_{j*}g_j^*$, we see that the two horizontal arrows in the right square of \eqref{diam:lem:pi-sheaf-glue-2-3} are the base change isomorphisms defined in \ref{lem:base-change-mor}. Hence, we obtain the first row of \eqref{diam:lem:pi-sheaf-glue-2}. Replacing the second row of \eqref{diam:lem:pi-sheaf-glue-1} by $h_k:U_i\times_X U_k\to U_k$, we obtain the second row of \eqref{diam:lem:pi-sheaf-glue-2} in the same way.
	\end{enumerate}
	
	Since
	\begin{align}
		\xymatrix{
			\{g_{jk}:U_i\times_XU_j\times_XU_k\to U_i\}_{j,k\in I}\ar@<0.6ex>[r]\ar@<-0.6ex>[r]& \{g_j:U_i\times_XU_j\to U_i\}_{j\in I}\ar[l]
		}
	\end{align}
	is a $1$-truncated \v Cech hypercovering of $U_i$, the equalizer corresponding to the first column in \eqref{diam:lem:pi-sheaf-glue-2} is equal to
	\begin{align}
		\ca{G}_i=\mrm{Eq}(\prod_{j\in I}g_{j*}g_j^*\ca{G}_i\rightrightarrows \prod_{j,k\in I}g_{jk*}g_{jk}^*\ca{G}_i)
	\end{align}
	by the sheaf property of $\ca{G}_i$ on $E_{/U_i}$. Since $\ca{F}_\bullet$ is $\pi$-Cartesian, the horizontal arrows in \eqref{diam:lem:pi-sheaf-glue-2} are $\pi^2$-isomorphisms by \ref{lem:pi-iso-retract} (see \cite[7.3]{he2022sen}). Therefore, the morphisms between the equalizers corresponding to each column in \eqref{diam:lem:pi-sheaf-glue-2} (see the second row of \eqref{diam:5.3.8} in the following) are $\pi^4$-isomorphisms. In order to show that the canonical map $f_i^*\ca{F}\to \ca{G}_i$ \eqref{eq:lem:pi-sheaf-glue-1} is a $\pi^8$-isomorphism, it remains to prove the square in the following natural diagram is commutative,
	\begin{align}\label{diam:5.3.8}
		\xymatrix{
			&\ca{G}_i\ar[d]_-{\iota}\ar@{=}[dl]&f_i^*f_{i*}\ca{G}_i\ar[l]\\
			\mrm{Eq}(\prod g_{j*}g_j^*\ca{G}_i\rightrightarrows \prod g_{jk*}g_{jk}^*\ca{G}_i)\ar[r]&\mrm{Eq}(\prod g_{j*}\ca{G}_{ij}\rightrightarrows \prod g_{jk*}\ca{G}_{ijk})&\mrm{Eq}(\prod f_i^*f_{j*}\ca{G}_j\rightrightarrows\prod f_i^*f_{jk*}\ca{G}_{jk})=f_i^*\ca{F}\ar[l]\ar[u]
		}
	\end{align}
	where $\iota$ is the natural map making the left triangle commutative, and in each equalizer, $j$ goes through $I$ for the first product, and $j,k$ go through $I$ for the second product. Consider the commutative diagram for any $j,k\in I$,
	\begin{align}
		\xymatrix{
			&U_i\times_XU_j\times_XU_k\ar[dd]_-{g_{jk}}\ar[dl]\ar[dr]&&\\
			U_i\times_XU_j\ar[dr]^-{g_j}&&U_i\times_XU_k\ar[dl]_-{g_k}\ar[r]^-{h_k}\ar[dr]^-{f_{ik}}&U_k\ar[d]^-{f_k}\\
			&U_i\ar[rr]^-{f_i}&&X
		}
	\end{align}
	from which we obtain the following natural commutative diagram.
	\begin{align}\label{diam:5.3.10}
		\xymatrix{
			&g_{j*}\ca{G}_{ij}\ar[dl]&\ca{G}_i\ar[dl]_-{\iota}\ar@{}[d]|-{(1)}&f_i^*f_{i*}\ca{G}_i\ar[dl]\ar[l]\\
			g_{jk*}\ca{G}_{ijk}&\mrm{Eq}(\prod g_{j*}\ca{G}_{ij}\rightrightarrows \prod g_{jk*}\ca{G}_{ijk})\ar[u]\ar[d]\ar[l] &f_i^*f_{ik*}\ca{G}_{ik}\ar[dl]_-{\jmath}&\mrm{Eq}(\prod f_i^*f_{j*}\ca{G}_j\rightrightarrows\prod f_i^*f_{jk*}\ca{G}_{jk})\ar[l]\ar[d]\ar[u]\\
			&g_{k*}\ca{G}_{ik}\ar[ul]&g_{k*}h_k^*\ca{G}_k\ar[l]\ar@{}[u]|-{(2)}& f_i^*f_{k*}\ca{G}_k\ar[ul]\ar[l]_-{\sim}
		}
	\end{align}
	Indeed, the natural map $\jmath:f_i^*f_{ik*}\ca{G}_{ik}=f_i^*f_{i*}g_{k*}\ca{G}_{ik}\to g_{k*}\ca{G}_{ik}$ is defined by applying the adjunction morphism $f_i^*f_{i*}\to \id$ to $g_{k*}\ca{G}_{ik}$, and other natural arrows have appeared in the diagrams \eqref{diam:lem:pi-sheaf-glue-2} and \eqref{diam:5.3.8}. Thus, the commutativity of (1) follows from applying the adjunction morphism $f_i^*f_{i*}\to \id$ to the canonical map $\ca{G}_i\to g_{k*}\ca{G}_{ik}$, and the commutativity of (2) follows from the following natural commutative diagram
	\begin{align}
		\xymatrix{
			g_{k*}\ca{G}_{ik}&f_i^*f_{i*}g_{k*}\ca{G}_{ik}\ar[l]_-{\jmath}&f_i^*f_{k*}h_{k*}\ca{G}_{ik}\ar@{=}[l]&\\
			g_{k*}h_k^*\ca{G}_k\ar[u]&f_i^*f_{i*}g_{k*}h_k^*\ca{G}_k\ar[l]\ar[u]&f_i^*f_{k*}h_{k*}h_k^*\ca{G}_k\ar@{=}[l]\ar[u]&f_i^*f_{k*}\ca{G}_k\ar[l]
		}
	\end{align}
	where the vertical arrows are induced by the canonical map $h_k^*\ca{G}_k\to \ca{G}_{ik}$, and the composition of the second row is the base change isomorphism $f_i^*f_{k*}\ca{G}_k\iso g_{k*}h_k^*\ca{G}_k$ by \ref{lem:base-change-mor} (we indeed used both two constructions of the base change isomorphism, see the construction of \eqref{diam:lem:pi-sheaf-glue-2}).
	In particular, we see that the natural diagram extracted from \eqref{diam:5.3.10},
	\begin{align}
		\xymatrix{
			\ca{G}_i\ar[d]_-{\iota}&f_i^*f_{i*}\ca{G}_i\ar[l]\\
			\mrm{Eq}(\prod g_{j*}\ca{G}_{ij}\rightrightarrows \prod g_{jk*}\ca{G}_{ijk})\ar[d]&\mrm{Eq}(\prod f_i^*f_{j*}\ca{G}_j\rightrightarrows\prod f_i^*f_{jk*}\ca{G}_{jk})\ar[d]\ar[u]\\
			g_{k*}\ca{G}_{ik}&f_i^*f_{k*}\ca{G}_k\ar[l]
		}
	\end{align}
	is commutative for any $k\in I$. This shows that the diagram \eqref{diam:5.3.8} is commutative, which completes the proof.
\end{proof}
\begin{myrem}\label{rem:pi-sheaf-glue}
	We expect a generalization to any $2$-truncated hypercovering $K_\bullet$ of $X$ as in \cite[\href{https://stacks.math.columbia.edu/tag/0D8E}{0D8E}]{stacks-project}.
\end{myrem}

\begin{myexample}\label{exmp:cech}
	Let $X$ be a quasi-compact and separated scheme, and let $K_0=\{U_i\to X\}_{0\leq i\leq k}$ be a finite open covering of $X$ consisting of affine open subschemes. For any $n\in\bb{N}$, we define a semi-representable object of the Zariski site $X_{\mrm{Zar}}$ of $X$,
	\begin{align}
		K_n=\{U_{i_0}\cap \cdots\cap U_{i_n}\to X\}_{0\leq i_0,\cdots, i_n\leq k}.
	\end{align}
	These $K_n$ naturally form a simplicial semi-representable object of $X_{\mrm{Zar}}$, $K_\bullet=(K_n)_{[n]\in\ob(\Delta)}$, called the \v Cech hypercovering associated to $K_0$ of $X$. We put
	\begin{align}
		X_n=\coprod_{0\leq i_0,\dots,i_n\leq k} U_{i_0}\cap \cdots\cap U_{i_n}
	\end{align}
	which is a finite disjoint union of affine open subschemes of $X$, and denote by $\nu_n:X_n\to X$ the canonical morphism. It is clear that the site $X_{\mrm{Zar}/K_n}$ is naturally equivalent to the Zariski site $X_{n,\mrm{Zar}}$. We also obtain a simplicial affine scheme $X_\bullet=(X_n)_{[n]\in\ob(\Delta)}$, and an augmentation $\nu:X_\bullet\to X$ (where we omit the subscript ``Zar'').
	
	For any $\ca{O}_{X_\bullet}$-module $\ca{F}_\bullet$, we consider the ordered \v Cech complex $\check{C}_{\mrm{ord}}^\bullet(X_\bullet,\ca{F}_\bullet)$, whose degree-$n$ term is the $R$-module (\cite[\href{https://stacks.math.columbia.edu/tag/01FG}{01FG}]{stacks-project})
	\begin{align}\label{eq:exmp:cech-3}
		\check{C}_{\mrm{ord}}^n(X_\bullet,\ca{F}_\bullet)=\prod_{0\leq i_0<\cdots<i_n\leq k} \ca{F}_n(U_{i_0}\cap \cdots\cap U_{i_n}).
	\end{align}
	In general, for any complex of $\ca{O}_{X_\bullet}$-modules $\ca{F}_\bullet^\bullet$, we consider the total complex of the ordered \v Cech complexes $\tot(\check{C}_{\mrm{ord}}^\bullet(X_\bullet,\ca{F}_\bullet^\bullet))$, whose degree-$n$ term is the $R$-module (see \cite[\href{https://stacks.math.columbia.edu/tag/01FP}{01FP}]{stacks-project})
	\begin{align}
		\tot^n(\check{C}_{\mrm{ord}}^\bullet(X_\bullet,\ca{F}_\bullet^\bullet))=\bigoplus_{p+q=n}\prod_{0\leq i_0<\cdots<i_p\leq k} \ca{F}_p^q(U_{i_0}\cap \cdots\cap U_{i_p}).
	\end{align}
	Indeed, it depends only on the $k$-truncation $(\ca{F}_n^\bullet)_{[n]\in\ob(\Delta_{\leq k})}$. If $\ca{F}_\bullet^\bullet$ is the pullback of a complex of quasi-coherent $\ca{O}_X$-module $\ca{F}^\bullet$ (i.e., $\ca{F}_\bullet^\bullet=\nu^*\ca{F}^\bullet$), then there is an isomorphism in the derived category $\dd(X)$ (see \cite[\href{https://stacks.math.columbia.edu/tag/01FK}{01FK}, \href{https://stacks.math.columbia.edu/tag/01FM}{01FM}, \href{https://stacks.math.columbia.edu/tag/0FLH}{0FLH}]{stacks-project}),
	\begin{align}\label{eq:cech-iso}
		\tot(\check{C}_{\mrm{ord}}^\bullet(X_\bullet,\ca{F}_\bullet^\bullet))\iso \rr\Gamma(X,\ca{F}^\bullet).
	\end{align}
\end{myexample}

\begin{mylem}\label{lem:cech-glue-sheaf}
	Under the assumptions in {\rm\ref{exmp:cech}} and with the same notation, for any quasi-coherent $\ca{O}_{X_\bullet}$-module $\ca{F}_\bullet$, $\ca{F}=\nu_*\ca{F}_\bullet$ is a quasi-coherent $\ca{O}_X$-module.
\end{mylem}
\begin{proof}
	By definition, $\ca{F}=\mrm{Eq}(\nu_{0*}\ca{F}_0\rightrightarrows \nu_{1*}\ca{F}_1)$. As $\nu_i$ is affine ($i=0,1$), $\nu_{i*}\ca{F}_i$ is a quasi-coherent $\ca{O}_X$-module. Hence, the equalizer $\ca{F}$ is also a quasi-coherent $\ca{O}_X$-module.
\end{proof}

\begin{myprop}\label{prop:cech-glue-sheaf}
	Under the assumptions in {\rm\ref{exmp:cech}} and with the same notation, let $r\in \bb{N}_{\geq 2}\cup\{\infty\}$, let $\mrm{sk}_r(X_\bullet)=(X_n)_{[n]\in\ob(\Delta_{\leq r})}$ be the $r$-truncation of $X_\bullet$, and let $\ca{F}_\bullet$ be a quasi-coherent $\ca{O}_{\mrm{sk}_r(X_\bullet)}$-module of finite type. Assume that $\ca{F}_\bullet$ is $\pi$-Cartesian (see {\rm\ref{para:pi-cartesian}}). Then, there exists a quasi-coherent $\ca{O}_X$-module $\ca{F}$ of finite type and a $\pi^{25}$-isomorphism $\nu^*\ca{F}\to\ca{F}_\bullet$.
\end{myprop}
\begin{proof}
	If we also denote by $\nu:\mrm{sk}_r(X_\bullet)\to X$ the augmentation, then $\ca{F}'=\nu_*\ca{F}_\bullet$ is a quasi-coherent $\ca{O}_X$-module by \ref{lem:cech-glue-sheaf} and the canonical morphism $\nu_0^*\ca{F}'\to \ca{F}_0$ is a $\pi^8$-isomorphism by \ref{prop:pi-sheaf-glue}. We claim that the canonical morphism $\nu^*\ca{F}'\to \ca{F}_\bullet$ is a $\pi^9$-isomorphism. Indeed, for any integer $0\leq n\leq r$ the morphism $\nu_n:X_n\to X$ is the composition of $\nu_0:X_0\to X$ with a projection $f:X_n\to X_0$ associated to a morphism $[0]\to [n]$ in $\Delta$. Since $f^*\ca{F}_0\to \ca{F}_n$ is a $\pi$-isomorphism by assumption, we see that $\nu_n^*\ca{F}'\to \ca{F}_n$ is a $\pi^9$-isomorphism.
	
	For any $0\leq i\leq k$, we denote by $\ca{G}_i$ the restriction of $\ca{F}_0$ to the component $U_i$, which is a quasi-coherent $\ca{O}_{X_i}$-module of finite type by assumption. By \ref{lem:pi-iso-retract}.(\ref{item:lem:pi-iso-retract-2}), there exists a $\pi^{16}$-isomorphism $\ca{G}_i\to \ca{F}'|_{U_i}$. Notice that $\ca{F}'$ is the filtered union of its quasi-coherent $\ca{O}_X$-submodules of finite type by \cite[6.9.9]{ega1-2}. Thus, there is a sufficiently large quasi-coherent $\ca{O}_X$-submodule $\ca{F}$ of finite type such that the $\pi^{16}$-surjection $\ca{G}_i\to \ca{F}'|_{U_i}$ factors through $\ca{F}|_{U_i}$ for any $0\leq i\leq k$. Thus, the inclusion $\ca{F}\subseteq \ca{F}'$ is $\pi^{16}$-surjective, which implies that the induced morphism $\nu^*\ca{F}\to \ca{F}_\bullet$ is a $\pi^{25}$-isomorphism.
\end{proof}

\begin{mylem}\label{lem:cech-glue-cohom-2}
	Under the assumptions in {\rm\ref{exmp:cech}} and with the same notation, let $\ca{F}_\bullet^\bullet\to \ca{G}_\bullet^\bullet$ be a $\pi$-isomorphism of complexes of quasi-coherent $\ca{O}_{\mrm{sk}_k(X_\bullet)}$-modules. Then, the map
	\begin{align}\label{eq:lem:cech-glue-cohom-2}
		\tot(\check{C}_{\mrm{ord}}^\bullet(X_\bullet,\ca{F}_\bullet^\bullet))\longrightarrow \tot(\check{C}_{\mrm{ord}}^\bullet(X_\bullet,\ca{G}_\bullet^\bullet))
	\end{align}
	is a $\pi$-isomorphism. In particular, it is a $\pi^2$-quasi-isomorphism.
\end{mylem}
\begin{proof}
	By taking sections on an affine scheme, $\ca{F}_p^q(U_{i_0}\cap \cdots\cap U_{i_p})\to \ca{G}_p^q(U_{i_0}\cap \cdots\cap U_{i_p})$ is still a $\pi$-isomorphism. This shows that \eqref{eq:lem:cech-glue-cohom-2} is a $\pi$-isomorphism. The second assertion follows from \ref{lem:pi-iso-retract}.
\end{proof}

\begin{mylem}\label{lem:cech-glue-cohom}
	Under the assumptions in {\rm\ref{exmp:cech}} and with the same notation, let $a$ be an integer, and let $\ca{F}_\bullet^\bullet\to \ca{G}_\bullet^\bullet$ be a morphism of complexes of quasi-coherent $\ca{O}_{\mrm{sk}_k(X_\bullet)}$-modules. Assume that for any $0\leq n\leq k$, the map $H^i(\ca{F}_n^\bullet)\to H^i(\ca{G}_n^\bullet)$ is a $\pi$-isomorphism for any $i> a$ and $\pi$-surjective for $i=a$. Then, the map
	\begin{align}\label{eq:lem:cech-glue-cohom-1}
		H^i(\tot(\check{C}_{\mrm{ord}}^\bullet(X_\bullet,\ca{F}_\bullet^\bullet)))\longrightarrow H^i(\tot(\check{C}_{\mrm{ord}}^\bullet(X_\bullet,\ca{G}_\bullet^\bullet)))
	\end{align}
	is a $\pi^{4(k+1)}$-isomorphism for any $i> a+k$ and $\pi^{2(k+1)}$-surjective for $i=a+k$.
\end{mylem}
\begin{proof}
	As the functor $\check{C}_{\mrm{ord}}^p(X_\bullet,-)$ \eqref{eq:exmp:cech-3} on quasi-coherent $\ca{O}_{\mrm{sk}_k(X_\bullet)}$-modules is exact, we see that $H^q(\check{C}_{\mrm{ord}}^p(X_\bullet,\ca{F}_\bullet^\bullet))=\check{C}_{\mrm{ord}}^p(X_\bullet,H^q(\ca{F}_\bullet^\bullet))$, where $H^q(\ca{F}_\bullet^\bullet)=(H^q(\ca{F}_n^\bullet))_{[n]\in\ob(\Delta_{\leq k})}$ is a quasi-coherent $\ca{O}_{\mrm{sk}_k(X_\bullet)}$-module. Thus, there is a spectral sequence
	\begin{align}\label{eq:lem:cech-glue-cohom}
		E_2^{pq}=H^p(\check{C}_{\mrm{ord}}^\bullet(X_\bullet,H^q(\ca{F}_\bullet^\bullet)))\Rightarrow H^{p+q}(\tot(\check{C}_{\mrm{ord}}^\bullet(X_\bullet,\ca{F}_\bullet^\bullet))),
	\end{align}
	which is convergent, since $\check{C}_{\mrm{ord}}^p(X_\bullet,\ca{F}_\bullet^q)=0$ unless $0\leq p\leq k$ (\cite[\href{https://stacks.math.columbia.edu/tag/0132}{0132}]{stacks-project}). 
	
	Let $\ca{K}_\bullet^\bullet$ be the cone of $\ca{F}_\bullet^\bullet\to \ca{G}_\bullet^\bullet$. The assumption implies that $H^i(\ca{K}_n^\bullet)$ is $\pi^2$-null for any $0\leq n\leq k$ and $i\geq a$. The convergent spectral sequence \eqref{eq:lem:cech-glue-cohom} for $\ca{K}_\bullet^\bullet$ implies that for any $i\in \bb{Z}$, there is a finite filtration of length $\leq (k+1)$ on $H^i(\tot(\check{C}_{\mrm{ord}}^\bullet(X_\bullet,\ca{K}_\bullet^\bullet)))$ whose graded pieces are subquotients of $E_2^{p,i-p}$ where $0\leq p\leq k$. Since $E_2^{p,i-p}$ is $\pi^2$-null for any $i\geq a+k$, we see that $H^i(\tot(\check{C}_{\mrm{ord}}^\bullet(X_\bullet,\ca{K}_\bullet^\bullet)))$ is $\pi^{2(k+1)}$-null for such $i$. The conclusion follows from the long exact sequence of cohomology groups associated to the distinguished triangle in $\mbf{K}(R)$,
	\begin{align}
		\tot(\check{C}_{\mrm{ord}}^\bullet(X_\bullet,\ca{F}_\bullet^\bullet))\longrightarrow \tot(\check{C}_{\mrm{ord}}^\bullet(X_\bullet,\ca{G}_\bullet^\bullet))\longrightarrow\tot(\check{C}_{\mrm{ord}}^\bullet(X_\bullet,\ca{K}_\bullet^\bullet))\longrightarrow\tot(\check{C}_{\mrm{ord}}^\bullet(X_\bullet,\ca{F}_\bullet^\bullet))[1].
	\end{align}
\end{proof}

\begin{myprop}\label{prop:pi-pseudo-coh-cech}
	Under the assumptions in {\rm\ref{exmp:cech}} and with the same notation, let $a\leq b$ be two integers, and let $\ca{F}_\bullet^\bullet$ be a complex of quasi-coherent $\ca{O}_{\mrm{sk}_k(X_\bullet)}$-modules vanishing in degrees $>b$. Assume that
	\begin{enumerate}
		\renewcommand{\labelenumi}{{\rm(\theenumi)}}
		\item $R$ is Noetherian, and that
		\item the $R$-module $H^p(\check{C}_{\mrm{ord}}^\bullet(X_\bullet,H^q(\ca{F}_\bullet^\bullet)))$ is of $\pi$-finite type for any $0\leq p\leq k$ and $q\geq a$ (see {\rm\ref{defn:pi-fini}}).
	\end{enumerate}
	Then, the complex of $R$-modules $\tot(\check{C}_{\mrm{ord}}^\bullet(X_\bullet,\ca{F}_\bullet^\bullet))$ is $\pi^l$-$[a+k,b+k]$-pseudo-coherent for an integer $l\geq 0$ depending only on $b-a$ and $k$.
\end{myprop}
\begin{proof}
	Consider the convergent spectral sequence \eqref{eq:lem:cech-glue-cohom}. Notice that for any $0\leq p\leq k$ and $q\geq a$, a subquotient of $E_2^{pq}=H^p(\check{C}_{\mrm{ord}}^\bullet(X_\bullet,H^q(\ca{F}_\bullet^\bullet)))$ is of $\pi$-finite type by \ref{lem:pi-pseudo-coh-noeth}.(\ref{item:lem:pi-pseudo-coh-noeth-2}). Since there is a finite filtration of length $\leq (k+1)$ on $H^i(\tot(\check{C}_{\mrm{ord}}^\bullet(X_\bullet,\ca{F}_\bullet^\bullet)))$ whose graded pieces are subquotients of $E_2^{p,i-p}$ where $0\leq p\leq k$, we see that $H^i(\tot(\check{C}_{\mrm{ord}}^\bullet(X_\bullet,\ca{F}_\bullet^\bullet)))$ is of $\pi^{k+1}$-finite type for any $i\geq a+k$ by \ref{lem:pi-pseudo-coh-noeth}.(\ref{item:lem:pi-pseudo-coh-noeth-2}). As $\tot(\check{C}_{\mrm{ord}}^\bullet(X_\bullet,\ca{F}_\bullet^\bullet))$ vanishes in degrees $>b+k$ by definition, the conclusion follows from \ref{prop:pi-pseudo-coh-noeth}.(\ref{item:prop:pi-pseudo-coh-noeth-1}).
\end{proof}

\section{Almost Coherent Modules}\label{sec:coh}
In this section, we fix a ring $R$ with an ideal $\ak{m}$ such that for any integer $l\geq 1$, the $l$-th powers of elements of $\ak{m}$ generate $\ak{m}$ (in particular, $\ak{m}=\ak{m}^2$, see \cite[2.1.6.(B)]{gabber2003almost}). Let $E$ be a site with a final object $*$, and let $\ca{O}$ be a sheaf of $R$-algebras on $E$. 

\begin{mydefn}[{\cite[2.7.3]{abbes2020suite}}]\label{defn:almost}
	 Let $M$ be an $\ca{O}$-module on $E$.
	\begin{enumerate}
		\renewcommand{\labelenumi}{{\rm(\theenumi)}}
		\item We say that $M$ is \emph{almost zero} if it is $\pi$-null for any $\pi\in\ak{m}$. We say that a morphism $f:M\to N$ of $\ca{O}$-modules is an \emph{almost isomorphism} if its kernel and cokernel are almost zero.
		\item We say that $M$ is \emph{of $\pi$-finite type} for some element $\pi\in R$ if there exists a covering $\{U_i\to *\}_{i\in I}$ in $E$ such that for any $i\in I$ there exist finitely many sections $s_1,\dots,s_n\in M(U_i)$ such that the induced morphism of $\ca{O}|_{U_i}$-modules $\ca{O}^{\oplus n}|_{U_i}\to M|_{U_i}$ has $\pi$-null cokernel. We say that $M$ is \emph{of almost finite type} if it is of $\pi$-finite type for any $\pi\in\ak{m}$.
		\item We say that $M$ is \emph{almost coherent} if $M$ is of almost finite type, and if for any object $U$ of $E$ and any finitely many sections $s_1,\dots,s_n\in M(U)$, the kernel of the induced morphism of $\ca{O}|_U$-modules $\ca{O}^{\oplus n}|_U\to M|_U$ is an $\ca{O}|_U$-module of almost finite type.
	\end{enumerate}
\end{mydefn}

We refer to Abbes-Gros \cite[2.7, 2.8]{abbes2020suite} for a more detailed study of almost coherent modules. They work in a slightly restricted basic setup for almost algebra \cite[2.6.1]{abbes2020suite}, but most of their arguments still work in our setup $(R,\ak{m})$ by adding the following lemmas.

\begin{mylem}\label{lem:pi-fini}
	Let $M$ be an $\ca{O}$-module on $E$, and let $\pi_1,\pi_2\in R$. If $M$ is of $\pi_i$-finite type for $i=1,2$, then it is of $(x\pi_1+y\pi_2)$-finite type for any $x,y\in R$. In particular, if there exists an integer $l\geq 1$ such that $M$ is of $\pi^l$-finite type for any $\pi\in \ak{m}$, then $M$ is of almost finite type.
\end{mylem}
\begin{proof}
	The problem is local on $E$. We may assume that there exist morphisms of $\ca{O}$-modules $f_i:\ca{O}^{\oplus n_i}\to M$ ($i=1,2$) with $\pi_i$-null cokernels. Thus, the cokernel of $f_1\oplus f_2:\ca{O}^{\oplus n_1}\oplus\ca{O}^{\oplus n_2} \to M$ is killed by $x\pi_1+y\pi_2$. The ``in particular'' part follows from the assumption that the ideal $\ak{m}$ is generated by the subset $\{\pi^l\ |\ \pi\in\ak{m}\}$.
\end{proof}

\begin{mylem}\label{lem:almost-iso-coh}
	Let $M$ be an $\ca{O}$-module.
	\begin{enumerate}
		\renewcommand{\labelenumi}{{\rm(\theenumi)}}
		\item Assume that there exists an integer $l\geq 1$ such that for any $\pi\in\ak{m}$, there exists an almost coherent $\ca{O}$-module $M_\pi$ and a $\pi^l$-isomorphism $M\to M_\pi$. Then, $M$ is almost coherent.\label{item:lem:almost-iso-coh-1}
		\item Assume that there exists an integer $l\geq 1$ such that for any $\pi\in\ak{m}$, there exists an almost coherent $\ca{O}$-module $M_\pi$ and a $\pi^l$-isomorphism $M_\pi\to M$. Then, $M$ is almost coherent.\label{item:lem:almost-iso-coh-2}
	\end{enumerate}
\end{mylem}
\begin{proof}
	(\ref{item:lem:almost-iso-coh-1}) The $\pi^l$-isomorphism $M\to M_{\pi}$ induces a $\pi^{2l}$-isomorphism $M_\pi\to M$ by \ref{lem:pi-iso-retract}. Such an argument shows that (\ref{item:lem:almost-iso-coh-1}) implies (\ref{item:lem:almost-iso-coh-2}). We also see that $M$ is of $\pi^{3l}$-finite type. Hence, $M$ is of almost finite type by \ref{lem:pi-fini}. For any object $U$ of $E$, and any morphism of $\ca{O}|_U$-modules $f:\ca{O}^{\oplus n}|_U\to M|_U$, consider the following commutative diagram.
	\begin{align}
		\xymatrix{
			0\ar[r]&\ke(f)\ar[d]\ar[r]&\ca{O}^{\oplus n}|_U\ar@{=}[d]\ar[r]^-{f}&M\ar[d]\\
			0\ar[r]&\ke(f_\pi)\ar[r]&\ca{O}^{\oplus n}|_U\ar[r]^-{f_\pi}&M_\pi
		}
	\end{align}
	It is clear that $\ke(f)\to \ke(f_\pi)$ is a $\pi^l$-isomorphism. Since $M_\pi$ is almost coherent by assumption, $\ke(f_\pi)$ is of almost finite type. Hence, $\ke(f)$ is of $\pi^{3l}$-finite type by the argument in the beginning. Thus, $\ke(f)$ is of almost finite type by \ref{lem:pi-fini}. This verifies the almost coherence of $M$.
\end{proof}

We collect some basic properties about almost coherence that will be used in the rest of this article. Their proofs are essentially given in \cite{abbes2020suite}, and we only give a brief sketch here.

\begin{myprop}[{\cite[2.7.16]{abbes2020suite}}]\label{prop:almost-coh-ext}
	Let $0\longrightarrow M_1\stackrel{u}{\longrightarrow} M_2\stackrel{v}{\longrightarrow} M_3\longrightarrow 0$ be an almost exact sequence of $\ca{O}$-modules on $E$. If two of $M_1,M_2,M_3$ are almost coherent, then so is the third.
\end{myprop}
\begin{proof}
	Since almost isomorphisms preserve almost coherence by \ref{lem:almost-iso-coh}, we may assume that $0\to M_1\to M_2\to M_3\to 0$ is exact.
	
	Assume that $M_2$ and $M_3$ are almost coherent. Then, $M_1$ is of almost finite type by \cite[2.7.14.(\luoma{3})]{abbes2020suite} and \ref{lem:pi-fini}. Hence, $M_1$ is almost coherent as a submodule of an almost coherent $\ca{O}$-module $M_2$ by definition.
	
	Assume that $M_1$ and $M_2$ are almost coherent. Then, $M_3$ is of almost finite type as a quotient of $M_2$. Let $U$ be an object of $E$. We need to show that any homomorphism $f_3:\ca{O}^{\oplus n}|_U\to M_3|_U$ has kernel of $\pi^2$-finite type for any $\pi\in\ak{m}$ by \ref{lem:pi-fini}. The problem is local on $E$. Thus, we may take a $\pi$-surjection $f_1:\ca{O}^{\oplus m}|_U\to M_1|_U$ and a lifting  $f_3':\ca{O}^{\oplus n}|_U\to M_2|_U$ of $f_3$. We put $f_2=(f_1,f_3'):\ca{O}^{\oplus m+n}|_U\to M_2|_U$, and obtain a morphism of short exact sequences
	\begin{align}
		\xymatrix{
			0\ar[r]&\ca{O}^{\oplus m}|_U\ar[r]\ar[d]^-{f_1}&\ca{O}^{\oplus m+n}|_U\ar[r]\ar[d]^-{f_2}&\ca{O}^{\oplus n}|_U\ar[r]\ar[d]^-{f_3}&0\\
			0\ar[r]&M_1|_U\ar[r]&M_2|_U\ar[r]&M_3|_U\ar[r]&0.
		}
	\end{align}
	The snake lemma shows that $\ke(f_2)\to \ke(f_3)$ is $\pi$-surjective. Since $\ke(f_2)$ is of almost finite type as $M_2$ is almost coherent, we see that $\ke(f_3)$ is of $\pi^2$-finite type.
	
	Assume that $M_1$ and $M_3$ are almost coherent. Then, $M_2$ is of almost finite type by \cite[2.7.14.(\luoma{2})]{abbes2020suite} and \ref{lem:pi-fini}. Let $U$ be an object of $E$. We need to show that any homomorphism $f_2:\ca{O}^{\oplus n}|_U\to M_2|_U$ has kernel of $\pi^2$-finite type for any $\pi\in\ak{m}$ by \ref{lem:pi-fini}. The problem is local on $E$. Thus, we may take a $\pi$-surjection $\ca{O}^{\oplus m}|_U\to \ke(v\circ f_2)$ as $\ke(v\circ f_2)$ is of almost finite type ($M_3$ is almost coherent). Thus, we obtain a commutative diagram
	\begin{align}
		\xymatrix{
			&\ca{O}^{\oplus m}|_U\ar[r]\ar[d]^-{f_1}&\ca{O}^{\oplus n}|_U\ar[r]\ar[d]^-{f_2}&M_3|_U\ar[r]\ar@{=}[d]&0\\
			0\ar[r]&M_1|_U\ar[r]^-{u}&M_2|_U\ar[r]^-{v}&M_3|_U\ar[r]&0.
		}
	\end{align}
	By snake lemma, we see that $\ke(f_1)\to \ke(f_2)$ is $\pi$-surjective. Since $\ke(f_1)$ is of almost finite type as $M_1$ is almost coherent, we see that $\ke(f_2)$ is of $\pi^2$-finite type.
\end{proof}

\begin{mycor}[{\cite[2.7.17]{abbes2020suite}}]\label{cor:almost-coh-ext}
	For any morphism $f:M\to N$ of almost coherent $\ca{O}$-modules, $\ke(f)$, $\im(f)$ and $\cok(f)$ are almost coherent.
\end{mycor}

\begin{mycor}\label{cor:almost-coh-coh}
	Assume that $\ca{O}$ is almost coherent as an $\ca{O}$-module. Then, any cohomology group of a complex of finite free $\ca{O}$-modules is almost coherent.
\end{mycor}
\begin{proof}
	A finite free $\ca{O}$-module is almost coherent by \ref{prop:almost-coh-ext}. Thus, a cohomology group of a complex of finite free $\ca{O}$-modules is almost coherent by \ref{cor:almost-coh-ext}.
\end{proof}

\begin{myprop}[{\cite[2.8.7]{abbes2020suite}}]\label{prop:almost-coh-aff}
	Let $X=\spec(A)$ be an affine scheme over $R$, let $\ca{F}$ be a quasi-coherent $\ca{O}_X$-module, and let $\pi\in\ak{m}$. Then, the $\ca{O}_X$-module $\ca{F}$ is of $\pi$-finite type (resp. of almost finite type, almost coherent) on the Zariski site of $X$ if and only if the $A$-module $\ca{F}(X)$ is of $\pi$-finite type (resp. of almost finite type, almost coherent) on the trivial site of a single point.
\end{myprop}
\begin{proof}
	It is clear that the statement for ``of $\pi$-finite type'' implies that for ``of almost finite type'' and thus implies that for ``almost coherent''. It remains to show that for any $\pi\in\ak{m}$ and any finitely many elements $f_1,\dots,f_n\in A$ generating $A$ as an ideal, an $A$-module $M$ is of $\pi$-finite type if and only if the $A_{f_i}$-module $M_{f_i}$ is of $\pi$-finite type for any $1\leq i\leq n$. The necessity is obvious. For the sufficiency, we write $M=\bigcup_{\lambda\in\Lambda}M_\lambda$ as a filtered union of its $A$-submodules of finite type. There exists $\lambda_0\in\Lambda$ large enough such that $M_{\lambda_0,f_i}\to M_{f_i}$ is a $\pi$-isomorphism for any $1\leq i\leq n$. Hence, $M_{\lambda_0}\to M$ is a $\pi$-isomorphism, which completes the proof.
\end{proof}

\begin{mylem}[{cf. \cite[2.2]{kiehl1972finite}}]\label{lem:almost-coh-resol}
	Let $k\in\bb{N}$, let $X_\bullet=(X_n)_{[n]\in\ob(\Delta_{\leq k})}$ be a $k$-truncated simplicial affine scheme over $R$, and let $M_\bullet$ be a quasi-coherent $\ca{O}_{X_\bullet}$-module. Assume that the $\ca{O}_{X_\bullet}$-modules $\ca{O}_{X_\bullet}$ and $M_\bullet$ are almost coherent. Then, for any $\pi\in \ak{m}$, there exists a $\pi$-exact sequence of quasi-coherent $\ca{O}_{X_\bullet}$-modules
	\begin{align}
		\cdots\longrightarrow F_\bullet^{-1}\longrightarrow F_\bullet^{0}\longrightarrow M_\bullet,
	\end{align}
	such that $F_n^i$ is a finite free $\ca{O}_{X_n}$-module for any $i\leq 0$ and $0\leq n\leq k$.
\end{mylem}
\begin{proof}
	Firstly, we construct $F_\bullet^0$. For each $0\leq n\leq k$, we take a finite free $\ca{O}_{X_n}$-module $N_n$ and a $\pi$-surjection $h_n:N_n\to M_n$ (as $M_n$ is of $\pi$-finite type, see \ref{prop:almost-coh-aff}). We put
	\begin{align}
		F_n^{0}=\bigoplus_{0\leq m\leq k}\bigoplus_{\alpha\in\ho_{\Delta}([m],[n])}\alpha^*N_m.
	\end{align}
	It forms naturally a finite free $\ca{O}_{X_\bullet}$-module $F_\bullet^{0}$. There is a natural morphism $F_n^{0}\to M_n$ defined on the $(m,\alpha)$-component by the composition
	\begin{align}
		\xymatrix{
			\alpha^*N_m\ar[r]^-{\alpha^*(h_m)}&\alpha^*M_m\ar[r]& M_n.
		}
	\end{align}
	It induces a $\pi$-surjective homomorphism of $\ca{O}_{X_\bullet}$-modules $F_\bullet^{0}\to M_\bullet$ (cf. the proof of \cite[2.2]{kiehl1972finite}). Notice that its kernel $M^{-1}_\bullet$ is also a quasi-coherent $\ca{O}_{X_\bullet}$-module which is almost coherent by \ref{cor:almost-coh-ext}. Thus, we can apply the previous procedure to $M^{-1}_\bullet$ and we construct $F_\bullet^{i}$ inductively for any $i\leq 0$.
\end{proof}

\section{Proof of the Main Theorem}\label{sec:proof}

This section is devoted to proving the following theorem.

\begin{mythm}\label{thm:main}
	Let $R$ be a ring with an ideal $\ak{m}$ such that for any integer $l\geq 1$, the $l$-th powers of elements of $\ak{m}$ generate $\ak{m}$. Consider a flat proper morphism of finite presentation $f:X\to S$ between $R$-schemes. Assume that $\ca{O}_X$ and $\ca{O}_S$ are almost coherent as modules over themselves. Then, for any quasi-coherent and almost coherent $\ca{O}_X$-module $\ca{M}$ and any $q\in\bb{N}$, $\rr^q f_*\ca{M}$ is a quasi-coherent and almost coherent $\ca{O}_S$-module.
\end{mythm}
\begin{proof}
	The problem is local on $S$. Thus, we may assume that $S=\spec(A)$ is affine. Since $f$ is quasi-compact and quasi-separated, $\rr^qf_*\ca{M}$ is a quasi-coherent $\ca{O}_S$-module for any $q\in\bb{N}$. Thus, it remains to prove that $H^q(X,\ca{M})$ is an almost coherent $A$-module by \ref{prop:almost-coh-aff}. We write $X$ as a finite union of affine open subschemes $X=\bigcup_{0\leq i\leq k}U_i$ for some $k\in \bb{N}_{\geq 2}$, and consider the $k$-truncated \v Cech hypercovering $X_\bullet=(X_n)_{[n]\in\ob(\Delta_{\leq k})}$ (see \ref{exmp:cech}), where for any $0\leq n\leq k$,
	\begin{align}
		X_n=\coprod_{0\leq i_0,\dots,i_n\leq k}U_{i_0}\cap\cdots\cap U_{i_n},
	\end{align}
	which is a finite disjoint union of affine open subschemes of $X$ as $X$ is separated. Let $\nu:X_\bullet\to X$ denote the augmentation. 
	
	We fix an element $\pi\in\ak{m}$ and a negative integer $a<-(k+2)$ in the following, and take a sequence of quasi-coherent $\ca{O}_{X_\bullet}$-modules by \ref{lem:almost-coh-resol},
	\begin{align}
		0\longrightarrow \ca{F}_\bullet^a\longrightarrow\cdots\longrightarrow \ca{F}_\bullet^{-1}\longrightarrow \ca{F}_\bullet^{0}\longrightarrow \ca{M}_\bullet=\nu^*\ca{M},
	\end{align}
	such that $\ca{F}^\bullet_n\to \ca{M}_n[0]$ is a $\pi$-$[a,0]$-pseudo resolution (see \ref{defn:pi-pseudo-coh}) for any $0\leq n\leq k$. In other words, for any $a\leq i\leq 0$ and $0\leq n\leq k$,
	\begin{enumerate}
		\renewcommand{\labelenumi}{{\rm(\theenumi)}}
		\item $\ca{F}^i_n$ is a finite free $\ca{O}_{X_n}$-module, and
		\item $H^i(\ca{F}^\bullet_n)$ is $\pi$-null for any $a<i<0$, and $H^0(\ca{F}^\bullet_n)\to \ca{M}_n$ is a $\pi$-isomorphism.
	\end{enumerate}
	For any morphism $\alpha:[m]\to [n]$ in $\Delta_{\leq k}$ (regarded also as a morphism $X_n\to X_{m}$), we denote by $C_\alpha^\bullet$ the cone of the induced map $\alpha^*\ca{F}_{m}^\bullet\to \ca{F}_n^\bullet$ of complexes of finite free $\ca{O}_{X_n}$-modules.
	\begin{mylem}\label{lem:main-1}
		For any morphism $\alpha:[m]\to [n]$ in $\Delta_{\leq k}$, there exists a homomorphism of finite free $\ca{O}_{X_n}$-modules $s_{\alpha}^i:C_\alpha^i\to C_\alpha^{i-1}$ for any $i\geq a+1$ such that 
		\begin{align}
			\pi^{-4a}\cdot \id_{C_\alpha^i}=s_{\alpha}^{i+1}\circ\df_{\alpha}^i+\df_{\alpha}^{i+1}\circ s_{\alpha}^i
		\end{align}
		for any $i\geq a+1$, where $\df_{\alpha}^i:C_\alpha^i\to C_\alpha^{i+1}$ is the differential map.
	\end{mylem}
	\begin{proof}
		We firstly note that $C_\alpha^i=\ca{F}_n^i\oplus \alpha^*\ca{F}_{m}^{i+1}$ is a finite free $\ca{O}_{X_n}$-module. In particular, $C_\alpha^\bullet$ vanishes outside $[a-1,0]$. By definition, $H^i(\ca{F}_n^\bullet)\to H^i(\ca{M}_n[0])$ is a $\pi$-isomorphism for any $i>a$. Notice that the induced map $\alpha^*H^i(\ca{M}_{m}[0])\to H^i(\ca{M}_n[0])$ is an isomorphism since $\ca{M}_\bullet=\nu^*\ca{M}$. Thus, the induced map $\alpha^*H^i(\ca{F}_{m}^\bullet)\to H^i(\ca{F}_n^\bullet)$ is a $\pi^2$-isomorphism of $\ca{O}_{X_n}$-modules for any $i>a$, which implies that $H^i(C_\alpha^\bullet)$ is $\pi^4$-null for any $i>a$. Thus, the conclusion follows directly from \ref{lem:pi-exact-homotopy} (see \eqref{eq:3.3.4}).
	\end{proof}
	Now we write $A$ as a filtered union of finitely generated $\bb{Z}$-subalgebras $A=\colim_{\lambda\in\Lambda}A_\lambda$. By \cite[8.5.2, 8.8.2, 8.10.5, 11.2.6]{ega4-3}, there exists an index $\lambda_0\in\Lambda$ such that $\pi\in A_{\lambda_0}$, 
	\begin{enumerate}
		\renewcommand{\labelenumi}{{\rm(\theenumi)}}
		\item a flat proper morphism of finite presentation $f_{\lambda_0}:X_{\lambda_0}\to S_{\lambda_0}=\spec(A_{\lambda_0})$ whose base change along $S\to S_{\lambda_0}$ is $f$,
		\item affine open subschemes $U_{\lambda_0,i}$ ($0\leq i\leq k$) of $X_{\lambda_0}$ whose base change along $S\to S_{\lambda_0}$ is $U_i$,
		\item a complex of finite free $\ca{O}_{X_{\lambda_0,\bullet}}$-modules $0\to\ca{F}_{\lambda_0,\bullet}^a\to\cdots\to \ca{F}_{\lambda_0,\bullet}^{-1}\to \ca{F}_{\lambda_0,\bullet}^{0}\to 0$ whose pullback along $S\to S_{\lambda_0}$ is $\ca{F}^\bullet_\bullet$ (where $X_{\lambda_0,\bullet}$ is the $k$-truncated \v Cech hypercovering associated to $X_{\lambda_0,0}=\coprod_{0\leq i\leq k}U_{\lambda_0,i}\to X_{\lambda_0}$),
		\item a homomorphism of finite free $\ca{O}_{X_{\lambda_0,n}}$-modules $s_{\lambda_0,\alpha}^i:C_{\lambda_0,\alpha}^i\to C_{\lambda_0,\alpha}^{i-1}$ for any $i\geq a+1$ and any morphism $\alpha:[m]\to [n]$ in $\Delta_{\leq k}$, such that
		\begin{align}\label{eq:main-2}
			\pi^{-4a}\cdot \id_{C_{\lambda_0,\alpha}^i}=s_{\lambda_0,\alpha}^{i+1}\circ\df_{\lambda_0,\alpha}^i+\df_{\lambda_0,\alpha}^{i+1}\circ s_{\lambda_0,\alpha}^i
		\end{align}
		for any $i\geq a+1$, where $C_{\lambda_0,\alpha}^\bullet$ is the cone of $\alpha^*\ca{F}_{\lambda_0,m}^\bullet\to \ca{F}_{\lambda_0,n}^\bullet$, and $\df_{\lambda_0,\alpha}^i:C_{\lambda_0,\alpha}^i\to C_{\lambda_0,\alpha}^{i+1}$ is the differential map.
	\end{enumerate}
	We note that $X_{\lambda_0}$ and $S_{\lambda_0}$ are Noetherian schemes.
	 \begin{mylem}\label{lem:main-2}
	 	For any morphism $\alpha:[m]\to [n]$ in $\Delta_{\leq k}$, the induced map of coherent $\ca{O}_{X_{\lambda_0,n}}$-modules $\alpha^*H^i(\ca{F}_{\lambda_0,m}^\bullet)\to H^i(\ca{F}_{\lambda_0,n}^\bullet)$ is a $\pi^{-4a}$-isomorphism for any $i\geq a+2$.
	 \end{mylem}
 	\begin{proof}
 		It follows directly from the relation \eqref{eq:main-2}. 
 	\end{proof}
 	\begin{mylem}\label{lem:main-3}
 		For any $i\geq a+2$, there exists a coherent $\ca{O}_{X_{\lambda_0}}$-module $\ca{G}_{\lambda_0}^i$ and a $\pi^{-100a}$-isomorphism $\nu_{\lambda_0}^*\ca{G}_{\lambda_0}^i\to H^i(\ca{F}_{\lambda_0,\bullet}^\bullet)$, where $\nu_{\lambda_0}:X_{\lambda_0,\bullet}\to X_{\lambda_0}$ is the augmentation.
 	\end{mylem}
 	\begin{proof}
 		It follows by applying directly \ref{prop:cech-glue-sheaf} to the coherent $\ca{O}_{X_{\lambda_0,\bullet}}$-module $H^i(\ca{F}_{\lambda_0,\bullet}^\bullet)$, whose condition is satisfied by \ref{lem:main-2}.
 	\end{proof}
 	\begin{mylem}\label{lem:main-4}
 		The $A_{\lambda_0}$-module $H^i(\check{C}_{\mrm{ord}}^\bullet(X_{\lambda_0,\bullet},H^j(\ca{F}_{\lambda_0,\bullet}^\bullet)))$ is of $\pi^{-200a}$-finite type for any $0\leq i\leq k$ and $j\geq a+2$.
 	\end{mylem}
	\begin{proof}
		Notice that by \eqref{eq:cech-iso} and \ref{lem:main-3}, we have
		\begin{align}
			H^i(\check{C}_{\mrm{ord}}^\bullet(X_{\lambda_0,\bullet},\nu_{\lambda_0}^*\ca{G}_{\lambda_0}^j))=H^i(X_{\lambda_0},\ca{G}_{\lambda_0}^j),
		\end{align}
		which is an $A_{\lambda_0}$-module of finite type, since $X_{\lambda_0}$ is proper over the Noetherian scheme $\spec(A_{\lambda_0})$. Thus, $H^i(\check{C}_{\mrm{ord}}^\bullet(X_{\lambda_0,\bullet},H^j(\ca{F}_{\lambda_0,\bullet}^\bullet)))$ is of $\pi^{-200a}$-finite type by \ref{lem:cech-glue-cohom-2} and \ref{lem:main-3}.
	\end{proof}
	\begin{mylem}\label{lem:main-5}
		The complex of $A_{\lambda_0}$-modules $\tot(\check{C}_{\mrm{ord}}^\bullet(X_{\lambda_0,\bullet},\ca{F}_{\lambda_0,\bullet}^\bullet))$ is $\pi^l$-$[a+k+2,k]$-pseudo-coherent for an integer $l\geq 0$ depending only on $a$ and $k$.
	\end{mylem}
	\begin{proof}
		It follows directly from \ref{prop:pi-pseudo-coh-cech} whose conditions are satisfied by \ref{lem:main-4}.
	\end{proof}
	\begin{mylem}\label{lem:main-6}
		The complex of $A$-modules $\tot(\check{C}_{\mrm{ord}}^\bullet(X_{\bullet},\ca{F}_{\bullet}^\bullet))$ is $\pi^l$-$[a+k+2,k]$-pseudo-coherent for an integer $l\geq 0$ depending only on $a$ and $k$.
	\end{mylem}
	\begin{proof}
		By \ref{prop:pi-pseudo-coh-bc} and \ref{lem:main-5}, $\tau^{\geq (a+k+2)}(A\otimes_{A_{\lambda_0}}^{\dl}\tot(\check{C}_{\mrm{ord}}^\bullet(X_{\lambda_0,\bullet},\ca{F}_{\lambda_0,\bullet}^\bullet)))$ is represented by a $\pi^l$-$[a+k+2,k]$-pseudo-coherent complex of $A$-modules for an integer $l\geq 0$ depending only on $a$ and $k$. Since $f_{\lambda_0}:X_{\lambda_0}\to S_{\lambda_0}$ is flat, $A\otimes_{A_{\lambda_0}}^{\dl}\tot(\check{C}_{\mrm{ord}}^\bullet(X_{\lambda_0,\bullet},\ca{F}_{\lambda_0,\bullet}^\bullet))$ is also represented by $A\otimes_{A_{\lambda_0}}\tot(\check{C}_{\mrm{ord}}^\bullet(X_{\lambda_0,\bullet},\ca{F}_{\lambda_0,\bullet}^\bullet))=\tot(\check{C}_{\mrm{ord}}^\bullet(X_{\bullet},\ca{F}_{\bullet}^\bullet))$. Hence, $\tau^{\geq (a+k+2)}(\tot(\check{C}_{\mrm{ord}}^\bullet(X_{\bullet},\ca{F}_{\bullet}^\bullet))$ is $\pi^l$-$[a+k+2,k]$-pseudo-coherent for an integer $l\geq 0$ depending only on $a$ and $k$ by \ref{prop:pi-pseudo-coh-derived}. The conclusion follows from applying \ref{lem:pi-pseudo-coh-iso}.(\ref{item:lem:pi-pseudo-coh-iso-2}) to $\tot(\check{C}_{\mrm{ord}}^\bullet(X_{\bullet},\ca{F}_{\bullet}^\bullet))\to\tau^{\geq (a+k+2)}(\tot(\check{C}_{\mrm{ord}}^\bullet(X_{\bullet},\ca{F}_{\bullet}^\bullet))$.
	\end{proof}
	\begin{mylem}\label{lem:main-7}
		The complex of $A$-modules $\check{C}_{\mrm{ord}}^\bullet(X_{\bullet},\ca{\ca{M}}_{\bullet})$ is $\pi^l$-$[a+k+2,k]$-pseudo-coherent for an integer $l\geq 0$ depending only on $a$ and $k$.
	\end{mylem}
	\begin{proof}
		Since $\ca{F}_\bullet^\bullet\to \ca{M}_\bullet[0]$ is a $\pi$-$[a,0]$-pseudo resolution, the map $H^i(\tot(\check{C}_{\mrm{ord}}^\bullet(X_{\bullet},\ca{F}_{\bullet}^\bullet)))\to H^i(\check{C}_{\mrm{ord}}^\bullet(X_{\bullet},\ca{\ca{M}}_{\bullet}))$ is a $\pi^{4(k+1)}$-isomorphism for any $i\geq a+k+1$ by \ref{lem:cech-glue-cohom}. The conclusion follows from \ref{lem:pi-pseudo-coh-iso}.(\ref{item:lem:pi-pseudo-coh-iso-1}) and \ref{lem:main-6}.
	\end{proof}

Recall that $\rr\Gamma(X,\ca{M})$ is represented by the ordered \v Cech complex $\check{C}_{\mrm{ord}}^\bullet(X_{\bullet},\ca{\ca{M}}_{\bullet})$ by \eqref{eq:cech-iso}. As we have taken $a<-(k+2)$ in the beginning, we see that for any $q\in\bb{N}$, $H^q(X,\ca{M})=H^q(\check{C}_{\mrm{ord}}^\bullet(X_{\bullet},\ca{\ca{M}}_{\bullet}))$ is $\pi^l$-isomorphic to the $q$-th cohomology $H'^q$ of a complex of finite free $A$-modules for an integer $l\geq 0$ depending only on $a$ and $k$ by \ref{lem:main-7} and \ref{defn:pi-pseudo-coh} (taking $a=-k-3$ at first is actually enough for this argument). Since $A$ is an almost coherent $A$-module by applying \ref{prop:almost-coh-aff} to $\ca{O}_S$, we see that $H'^q$ is an almost coherent $A$-module by applying \ref{cor:almost-coh-coh} to the trivial site of a single point with structural sheaf given by $A$. Since $l$ is independent of the choice of $\pi\in\ak{m}$, the $A$-module $H^q(X,\ca{M})$ is almost coherent by \ref{lem:almost-iso-coh}, which completes the proof of our main theorem \ref{thm:main}.
\end{proof}

\section{Remark on Abbes-Gros' Construction of the Relative Hodge-Tate Spectral Sequence}\label{sec:remark}

\begin{mypara}
	Let $K$ be a complete discrete valuation field of characteristic $0$ with algebraically closed residue field of characteristic $p>0$. Let $(f,g):(X'^{\triangleright}\to X')\to (X^\circ\to X)$ be a morphism of open immersions of quasi-compact and quasi-separated schemes over $\spec(K)\to \spec(\ca{O}_K)$. Consider the following conditions:
	\begin{enumerate}
		\renewcommand{\labelenumi}{{\rm(\theenumi)}}
		\item The associated log schemes $(X',\scr{M}_{X'})$ and $(X,\scr{M}_X)$ endowed with the compactifying log structures are adequate in the sense of \cite[\Luoma{3}.4.7]{abbes2016p} (which holds for instance if the open immersions $X'^{\triangleright}\to X'$ and $X^\circ\to X$ are semi-stable over $\spec(K)\to \spec(\ca{O}_K)$, see \cite[5.11]{he2024falmain}).\label{remark:1}
		\item The morphism of log schemes $(X',\scr{M}_{X'})\to (X,\scr{M}_X)$ is smooth and saturated.\label{remark:2}
		\item The morphism of schemes $g:X'\to X$ is projective.\label{remark:3}
	\end{enumerate}
	Under these assumptions, Abbes-Gros proved Faltings' main $p$-adic comparison theorem in the relative case for the morphism $(f,g)$ \cite[5.7.4]{abbes2020suite}, and constructed a relative Hodge-Tate spectral sequence \cite[6.7.5]{abbes2020suite} (for an explicit local version, see \cite[6.9.6]{abbes2020suite} and \cite[1.4]{he2021coh}). We explain that their proof and construction are still valid if we replace the assumption (\ref{remark:3}) by the following assumption:
	\begin{enumerate}
		\renewcommand{\labelenumi}{{\rm(\theenumi)}}
		\item[(3)'] The morphism of schemes $g:X'\to X$ is proper.
	\end{enumerate}
\end{mypara}

\begin{mypara}
	The assumption on the projectivity of $g$ has been only used in the proof of \cite[5.3.31]{abbes2020suite}. There, they encountered a Cartesian diagram of schemes
	\begin{align}
		\xymatrix{
			\overline{X}'^{(\infty)}\ar[d]_-{g^{(\infty)}}\ar[r]&X'\ar[d]^-{g}\\
			\overline{X}^{(\infty)}\ar[r]&X
		}
	\end{align}
	where $\overline{X}^{(\infty)}$ is an $\ca{O}_{\overline{K}}$-scheme such that $\ca{O}_{\overline{X}^{(\infty)}}$ and $\ca{O}_{\overline{X}'^{(\infty)}}$ are almost coherent as modules over themselves (\cite[5.3.5.(\luoma{2})]{abbes2020suite}), and a quasi-coherent and almost coherent $\ca{O}_{\overline{X}'^{(\infty)}}$-module $\scr{G}$. For proving the almost coherence of $\rr^ig^{(\infty)}_*\scr{G}$, they applied \cite[2.8.18]{abbes2020suite} where the assumption on the projectivity of $g$ had been used. 
	
	Now we replace the assumption (\ref{remark:3}) by the assumption (3)', by replacing \cite[2.8.18]{abbes2020suite} by our main theorem \ref{thm:main}. Indeed, the morphism $g$ is flat by the assumption (\ref{remark:2}) (see \cite[4.5]{kato1989log}), proper by the assumption (3)', of finite presentation by the assumptions (3)' and (\ref{remark:1}) (as $X$ is locally Noetherian). Hence, so is the morphism $g^{(\infty)}$ by base change. Therefore, we deduce the almost coherence of $\rr^ig^{(\infty)}_*\scr{G}$ from our main theorem \ref{thm:main}.
\end{mypara}

\bibliographystyle{myalpha}
\bibliography{bibli}


\end{document}